\documentclass[11pt,oneside,reqno]{amsart}
\usepackage[all]{xy}
\usepackage{amsfonts,amsmath,amssymb,amscd}
\usepackage[breaklinks]{hyperref}
\usepackage{caption} 
\usepackage{mathrsfs}
\usepackage{color}
\usepackage{amsmath}

\definecolor{blue}{rgb}{0,0,1}
\definecolor{red}{rgb}{1,0,0}
\definecolor{green}{rgb}{0,.6,.2}
\definecolor{purple}{rgb}{1,0,1}

\long\def\red#1\endred{{\color{red}#1}}
\long\def\blue#1\endblue{{\color{blue}#1}}
\long\def\purple#1\endpurple{{\color{purple}#1}}
\long\def\green#1\endgreen{{\color{green}#1}}

\allowdisplaybreaks
 
\newcommand{\sm}{\left(\begin{smallmatrix}}
\newcommand{\esm}{\end{smallmatrix}\right)}

\renewcommand{\(}{\left\(}
\renewcommand{\)}{\right\)}
\renewcommand{\[}{\left\[}
\renewcommand{\]}{\right\]}

\numberwithin{equation}{section}

\newtheorem{theorem}{Theorem}
\newtheorem{proposition}{Proposition}

\newtheorem{lemma}{{\bf Lemma}}
\newtheorem{coro}[subsection]{{\bf Corollary}}
\newtheorem{definition}{Definition}

\newtheorem{remark}{Remark}
\newtheorem{conjecture}[subsection]{Conjecture}
\newtheorem{example}{Example}

\numberwithin{theorem}{section}
\numberwithin{corollary}{section}
\numberwithin{example}{section}
\numberwithin{proposition}{section}
\numberwithin{remark}{section}
\numberwithin{definition}{section}
\numberwithin{lemma}{section}
\allowdisplaybreaks

\begin{document}
\title[Kronecker second limit formula for real quadratic fields]{ Kronecker second limit formula for real quadratic fields}

%%%%%%%%%%%%%%%%%%%%%%%%%%%%%%%%%%%%%%%%%%%%%%%
\author{YoungJu Choie}
\address{Department of Mathematics, Pohang University of Science and Technology \\
Pohang, Republic of Korea}
\email{yjc@postech.ac.kr}
\author{Rahul Kumar}
\address{Department of Mathematics, Indian Institute of Technology, Roorkee-247667, Uttarakhand, India}
\email{rahul.kumar@ma.iitr.ac.in} 

%\thanks{The first author was partially supported by   NRF$2022R1A2B5B0100187111 , 2021R1A6A1A1004294412$ and the second author  was partially supported by the grant  IBS-R$003$-D1  of the IBS-CGP, POSTECH, South Korea and the Fulbright-Nehru Postdoctoral Fellowship Grant 2846/FNPDR/2022.
%}

%\thanks{The first author was partially supported by $NRF 2018R1A4A 1023590, NRF 2017R1A2B2001807$ and }

  %\thanks{2010 Mathematics Subject Classification: Primary Secondary   \\
%Keywords: Herglotz function, Kronecker limit formula}
\maketitle
\pagenumbering{arabic}
\pagestyle{headings}
\begin{abstract}
In this paper, the second Kronecker ``limit" formula for a real quadratic field is established for the first time. More precisely, we obtain the second Kronecker limit formula of Zagier's zeta function. Using   the   reduction theory of Zagier, which connects   Zagier's zeta function to the zeta function of real quadratic fields,  we express the values of the zeta function of narrow ideal classes in real quadratic fields at natural arguments in terms of an analytic function which we call the \emph{higher Herglotz-Zagier-Novikov function} and denote it by $\mathscr{F}_k(x; \alpha, \beta)$. This function plays a central role in our study. The function $\mathscr{F}_k(x; \alpha, \beta)$ possesses elegant properties, for example, we prove that it satisfies the  two, three and six-term functional equations. As a result of our Kronecker limit formula and functional equations, we provide another expression for the combinations of zeta values. Finally, we interpret  our  Kronecker ``limit" formula in terms of cohomological relations and establish a connection between $\mathscr{F}_k(x; \alpha, \beta)$  and a generalized Dedekind-eta function.
\end{abstract}
\tableofcontents

\today

%%%%%%%%%%%%%%%%%%%%%%%
\section{Introduction and motivation}
%%%%%%%%%%%%%%%%%%%%%%%%%%

The classical first and second limit formulas of Kronecker have many applications, not only in number theory but also in   physics.   The Kronecker limit formulas are concerned with the constant term in the Laurent series expansion of certain Dirichlet series at $s=1$. Kronecker \cite{Kro} derived such a formula for the Dedekind zeta function over imaginary quadratic fields for the first time. He also provided several applications of his first limit formula, one of which is a solution to Pell's equation. In his lectures \cite{Se},  Siegel gives  a delightful explanation of   the   Kronecker limit formulas and their significance in analytic and algebraic number theory, while Weil \cite{weil} traces their historical developments. These two formulas play a vital role in the work of Stark \cite{stark1}, who studied the values of $L$-functions at $s=1$ and formulated his namesake conjecture, \emph{Stark's conjecture}. Stark further explored this topic in his other papers in this series \cite{stark2, stark3, stark4}.
 
In his famous paper \cite{Z},  Zagier developed the first Kronecker limit formula for real quadratic fields. In his work, he encountered a universal function, later called    the   \emph{Herglotz function} in \cite{RZ}
\begin{align}\label{her int}
F(x)=\int_0^1\left(\frac{1}{1-t}+\frac{1}{\log{t}}\right) \log(1-t^x) \frac{dt}{t},\quad  \mathrm{Re}(x)>0.
\end{align}
Vlasenko and Zagier \cite{VZ} studied    the    \emph{higher Herglotz function} $\mathscr{F}_k(x)$ to derive    a    higher Kronecker ``limit" formula for real quadratic fields (see \eqref{hk} below for the definition of $\mathscr{F}_k(x)$). 
%Moreover, Novikov \cite{novikov} also proved   variations on the Kronecker limit formula for real quadratic fields involving special values of the function, for $\alpha, \beta \in \mathbb{R}, \beta \notin \mathbb{Z},$
%\begin{eqnarray}\label{nov}
%\rho(x, \alpha, \beta) =\int_0^{1}\frac{\ln{(1-t^x e^{2\pi i \alpha})}}{e^{-2\pi i \beta}-t}  dt, \quad  \mathrm{Re}(x)>0.
%\end{eqnarray} 
Prior to Zagier, the first Kronecker limit formulas for real quadratic fields were discovered by Hecke \cite{hecke} and Herglotz \cite{H}. 
%Bump and Goldfeld \cite{bg} established similar formula for cubic fields and mentioned that:

%``\emph{Much work remains to be
%done in this direction, and one can only begin to see a whole new world of limit formulae
%emerging into view}".

 While there has been extensive development for the first Kronecker limit formula and its applications (see \cite{bg, Cho-Sel, hecke, H, novikov, ram, sh, sh2, sh1, Z}),   the    literature on the second Kronecker  limit formula is relatively limited. Apart from Kronecker's original work \cite{Kro} and references like \cite{ram,Se}, recent contributions include Kopp's paper \cite{kopp} for indefinite zeta functions (also covered in his PhD thesis \cite{kopp thesis}). Our primary focus in this paper is on the second Kronecker limit formula for real quadratic fields.

In this paper, we obtain Kronecker's second  and higher ``limit" formulas for real quadratic fields for the first time. More precisely, we obtain the second Kronecker limit formula of the Zagier's zeta function, which is defined in \eqref{zagier zeta function} below. Using   the   reduction theory of Zagier, which connects  the Zagier's zeta function $\mathcal{Z}(s,(\alpha,\beta),\mathscr{B})$ to the zeta function of real quadratic fields $\zeta(s,(\alpha,\beta),\mathscr{B})$,  we express the values of the zeta function of narrow ideal classes in real quadratic fields at natural arguments in terms of an analytic function $\mathscr{F}_k(x; \alpha, \beta)$ which we call the \emph{higher Herglotz-Zagier-Novikov function} (defined in \eqref{int repres for hh} below).  This function plays a pivotal role in our study and is  an extension of  the Herglotz function $F(x)$, introduced by Zagier \cite{Z} to study the first Kronecker limit formula of real quadratic fields. Moreover, the higher Herglotz function $\mathscr{F}_k(x)$, which was studied by Vlasenko and Zagier in \cite{VZ} to study higher limit formulas, is also a limiting case of $\mathscr{F}_k(x;\alpha,\beta)$ as $\alpha, \beta\to 0$ up to certain correction terms (see Proposition \ref{special cases} below).
%In \cite{Z} and \cite{VZ} Zagier, Vlasenko and Zagier  Similar to the work of Zagier \cite{Z}, and Vlasenko and Zagier \cite{VZ}, we also stumbled upon a Herglotz-type function which we named as \emph{higher Herglotz-Zagier-Novikov} function and denoted by $\mathscr{F}_k(x;\alpha,\beta)$, see Section \ref{our function} below for more details. 

%Since the function $\mathscr{F}_k(x;\alpha,\beta)$ occurs for the first time in this paper, 

We have also examined  properties of $ \mathscr{F}_k(x;\alpha,\beta)$ such as two-term, three-term and six-term  functional equations, asymptotic expansions, etc,    in this paper. As applications of our Kronecker limit formula and functional equations, we provide   an    interpretation of our limit formulas in terms of cohomological relations, and evaluate rational zeta values.

%For a survey in this area of research, we refer the reader to \cite{kopp thesis}.
\medskip

%%%%%%%%%%%%%%%%%%%%%%%%%%
%%%%%%%%%%%%%%%%%%%%%%%%%%%%%%%%%%%%%%%%%%
\section{The second Kronecker limit formula}\label{appl to klf}
%%%%%%%%%%%%%%%%%%%%%%%%%%%%%

Consider the Dedekind zeta function $\zeta_K(s)$ of   a   number field $K$:
$$\zeta_K(s)=\sum_A\zeta(s,A),$$ where
$A$ runs over the ideal class group of $K $ and 
\begin{eqnarray*}
\zeta(s, A)=\sum_{a\in A}\frac{1}{N(a)^s}\, \, \,\, (\mathrm{Re}(s)>1).\end{eqnarray*}

It is well known   that  
\begin{enumerate}
\item $\zeta_K(s)$ converges absolutely and uniformly when $\mathrm{Re}(s)>1.$

\item $\zeta_K(s)$ has an analytic continuation to $\mathbb{C}$ except at $s=1 ,$ where it has a simple pole.

\item It satisfies   a   functional equation
 relating $s$ to $(1-s).$
 
\end{enumerate}

Consider the Laurent expansion of $\zeta(s, A)$ at $s=1$ 
$$\zeta(s,A)=\frac{\kappa}{s-1}+ {\varrho}(A) +\varrho_1(A)(s-1)+\cdots.$$

%When $K$ is an imaginary quadratic field, the value  of $\rho(A)$ was evaluated  for the first time by Kronecker \cite{Kro}; hence it is now known as ``Kronecker first limit formula".  The ``Kronecker second  limit formula" is concerned about the following zeta function \cite{Se} 
The value of $\varrho(A)$ for $K$ being an imaginary quadratic field was initially computed by Kronecker \cite{Kro}. As a result, it is commonly referred to as the \textit{first Kronecker limit formula}. See \cite{Se}, \cite{Z} for more details. 
%  
%It is well known (see \cite{Z}) that there is a bijection between a set of ideals in $A$ and the set of principal ideals  in the following way:  
%for a fixed ideal $\underline{b}$ in the class $A^{-1},$ the correspondence 
%\begin{eqnarray}\label{co}
%\mathfrak{a} \longrightarrow \mathfrak{a}\underline{b}=\mbox{some principal ideal $(\lambda)$}, \lambda\in \underline{b}.
%\end{eqnarray} 
%Two numbers $\lambda_1, \lambda_2\in \underline{b}$ define the same principal ideal iff $\lambda_1=\epsilon \lambda_2$ for unit $\epsilon $ iff they have the same image in $\underline{b}/U $ with a group of units in $K.$
%Hence
%\begin{eqnarray}\label{ide}
%\zeta(s,A)=N(\underline{b})^s\sum_{a\in A}\frac{1}{N(a\underline{b})^s}
%=N(\underline{b})^s\sum_{\lambda \in \underline{b} / U}
%\frac{1}{N(\lambda)^s} \end{eqnarray}
% 
%%%%%%%%%%%%%%%%%%%%%%%%%%%%%%%%%%%%%
%\subsection{Kronecker limit formula}
%%%%%%%%%%%%%%%%%%%%%%%%%%%%%%%%%%%%%%
%\subsection{When $K$ is an imaginary quadratic field :}
%  If $K$ is an imaginary quadratic field of discriminant $D<0, D\neq -3, -4,$ then 
%it has an order $2$ unit group $U.$ 
%Assume that $\overline{b}$ has the oriented basis $\{w,1\}, $
%that is, $Im(w)>0$ if $K$ is an imaginary quadratic field. From (\ref{ide}) we get
%$$\zeta(s, A)=\frac{|D|^{-\frac{s}{2}}}{|U|}
%\sum_{\tiny{
% \begin{array}{cc} p,q\in \mathbb{Z} \\
% (p,q)\neq (0,0) \end{array}}}
%\frac{1}{Q(p,q)^s},$$
%where
%$$Q(p,q)= \frac{|pw+q|^2}{2Im(w)} $$  is the binary form $N(pw+q)$   normalized to have determinant $-1.$
% 
Meanwhile, the \textit{second Kronecker  limit formula} for an imaginary quadratic field pertains to the zeta function \cite{Se}\footnote{{The first Kronecker  limit formula essentially deals with the case where $\alpha$ and $\beta$ in \eqref{qzeta1} are integers. In this case, the zeta function $\zeta_Q(s, (\alpha, \beta))$  is nothing but the classical real analytic Eisenstein series and it  reduces to the partial zeta function $\zeta(s,A)$ for an imaginary quadratic field. For further details, refer to \cite{Se} or \cite{Z}.}} 
\begin{align}\label{qzeta1}
\zeta(s,\tau, (\alpha, \beta))=\mathrm{Im}(\tau)^s{\sum_{p,q\in \mathbb{Z}}}^{'}\frac{e^{2\pi i (p\alpha+q\beta)}}{|p\tau+q|^{2s}},\  \mathrm{Im}(\tau)>0,\ (\alpha,\beta)\in\mathbb{R}^2\backslash\mathbb{Z}^2, 
\end{align}
where   the   summation runs over all ordered pairs of integers $(p,q)$ except $(0,0)$.  The zeta function $\zeta(s,\tau, (\alpha, \beta))$, also referred to as the twisted real analytic Eisenstein series, is an analytic function in the half-plane Re$(s)>1/2$. The second  Kronecker limit formula is then expressed as follows \cite{Kro}, \cite[p.~40, Theorem 2]{Se}:  
\begin{align}\label{2nd imaginary}
\zeta(1,\tau, (\alpha, \beta)) 
=2\pi^2  \beta^2  \mathrm{Im}( \tau )-2\pi \log{  \left|
 \frac{\vartheta_1(  z_\tau, \tau )}{\eta(\tau)} \right| },
\end{align}
where 
% $\tau \in \mathbb{H}:=\{z\in \mathbb{C} : Im(z)>0\} ,\  
 $z_\tau=\alpha-\beta \tau \in \mathbb{C}$ and $\eta(\tau)$ is the Dedekind eta function 
\begin{align*}
\eta(\tau)&=q^{\frac{1}{24}}\prod_{n\geq1}\bigl(1-q^n\bigr),\ q=e^{2\pi i \tau}, \  \text{Im}(\tau)>0 
\end{align*}
and $\vartheta_1(z,  \tau)$ is the elliptic theta-function defined as \cite[p.~40]{Se}
\begin{align*}
\vartheta_1(z, \tau) &=-i\xi^{\frac{1}{8}}\left(\xi^{\frac{1}{2}}-\xi^{-\frac{1}{2}}\right)  \prod_{m\geq 1}(1-q^m\xi )(1-q^{ m}\xi^{-1}   )(1-q^m),\ 
 \xi=e^{2\pi i z},\ z\in\mathbb{C}.
\end{align*} 

%%%%%%%%%%%%%%%%%%%%%%%%%%%%%
%\subsection{When $K$ is a real quadratic field}
 %  
 
On the other hand, when $K$ is a real quadratic field, Hecke \cite{hecke} was   the    pioneer who derived the  first Kronecker  limit formula  in terms of the integral of the Dedekind eta function $\log|\eta(\tau)|.$ In $1975$,  Zagier \cite{Z}  found a more explicit formula compared to Hecke's for the real quadratic case. His formula involves a universal function known as the Herglotz function $F(x)$, as defined in \eqref{her int}. This result was further generalized  to evaluate the higher zeta values  and to give   a    cohomological interpretation of the  higher Kronecker limit formula in \cite{VZ}. However, to the best of our knowledge, a second Kronecker limit formula for real quadratic fields, serving as the counterpart to \eqref{2nd imaginary} in the imaginary quadratic field setting, is absent from the existing literature. In what follows, we aim to address this gap.

In order to establish our result on the second Kronecker limit formula for real quadratic fields, it is necessary to provide a summary of key concepts and background knowledge. For further details, we refer the reader to Zagier's influential paper \cite{Z}, which offers a comprehensive treatment of these concepts. It is well known that there is a bijection between a set of ideals in $A$ and the set of principal ideals  in the following way:  
for a fixed ideal $\underline{b}$ in the class $A^{-1},$ the correspondence 
\begin{eqnarray*}
\mathfrak{a} \longrightarrow \mathfrak{a}\underline{b}=\mbox{some principal ideal $(\lambda)$}, \lambda\in \underline{b}.
\end{eqnarray*} 
The numbers $\lambda_1$ and $\lambda_2$ in the ideal $\underline{b}$ define the same principal ideal if and only if $\lambda_1$ can be written as $\epsilon \lambda_2$, where $\epsilon$ is a unit. This is equivalent to saying that $\lambda_1$ and $\lambda_2$ have the same image in $\underline{b}/U$, where $U$ is the group of units in the ring of integers $O_K$. Hence, the partial zeta function can be expressed as follows:
\begin{eqnarray}\label{ide}
\zeta(s,A)=N(\underline{b})^s\sum_{a\in A}\frac{1}{N(a\underline{b})^s}
=N(\underline{b})^s\sideset{}{'}\sum_{\lambda \in \underline{b} / U}
\frac{1}{|N(\lambda)|^s},
 \end{eqnarray}
where the prime on the summation indicates that the term with a value of $0$ is to be excluded. When $K=\mathbb{Q}(\sqrt{D})$ is a real quadratic field with discriminant $D>0,$ 
let $\epsilon$ be a fundamental unit (the smallest unit $>1,$ where we have fixed once and for all an embedding $K\subset \mathbb{R}).$ Then $U=\{\pm \epsilon^n :  n\in \mathbb{Z}\}$ so that
(\ref{ide}) gives 
$$2 \zeta(s, A)=N(\underline{b})^s\sideset{}{'}\sum_{\lambda\in \underline{b} / \epsilon} \frac{1}{|\lambda \lambda'|^s}.$$
Further, assuming that $K$ does not contain any unit of negative norm, it is known that each ideal class $A$ can be represented as the union of two disjoint narrow ideal classes, denoted as $A = \mathscr{B} \cup \mathscr{B}^*$. Here, $\mathscr{B}$ and $\mathscr{B}^*$ are narrow ideal classes, where two ideals $a$ and $\underline{b}$ belong to the same narrow ideal class if there exists a principal ideal $(\alpha)$ with $N(\alpha) > 0$ such that $a = (\alpha)\underline{b}$. Notably, $\mathscr{B}^* = \Theta\mathscr{B}$, where $\Theta$ represents the narrow ideal class of principal ideals $(\alpha)$ with $N(\alpha) < 0$.

%From \eqref{co1}    there is  a bijective map  between
%\begin{eqnarray}\label{coo}
%\{ \mbox{ideal \, \, } \mathfrak{a}\in \mathscr{B} \} \leftrightarrow 
%\{\mathfrak{a}\underline{b}= \mbox{some principal ideal $(\lambda)$}
%\}.
%\end{eqnarray}
%where $\lambda \in \underline{b}$ with a given ideal $\underline{b} \in \mathscr{B}^{-1}.$

Next,  recall the relation between ideal classes and reduced numbers (page 163 in \cite{Z}):
a reduced number $w\in K $ satisfies
$$ w>1 > w' >0$$
and  has a pure periodic continued fraction expansion.  If  $\mathscr{B}$ is a narrow ideal class of $K$ with length  $r=\ell(\mathscr{B})$  and cycle $((b_1, \cdots, b_r)),$   where $ b_i\in \mathbb{Z}, b_i \geq 2,   $ there are exactly $r$ many reduced numbers $w\in K$ for which $\{1,w\}$ is a  basis for some ideal in $\mathscr{B},$ say the numbers
 
\begin{eqnarray} \label{cont}
w_k =  b_k-\cfrac{1}{b_{k+1}-\ \genfrac{}{}{0pt}{0}{}{\ddots\genfrac{}{}{0pt}{0}{}\ \genfrac{}{}{0pt}{0}{}{-\cfrac{1}{b_{r}-\cfrac{1}{b_{ 1}-\genfrac{}{}{0pt}{0}{}{\ddots}}}}}} 
\end{eqnarray}
 for $k=1, \cdots, r.$  Let 
\begin{align}\label{redB}
\mathrm{Red}(\mathscr{B}):=\{w_1, w_2, \cdots, w_r \}
\end{align} 
be the set of all reduced forms in $\mathscr{B}.$  We may extend   the   definition to all $k$ by setting $w_k$   to  depend   only on 
$k\pmod{r}.$
Fix  the ideal $\underline{b}=\mathbb{Z}w_0+\mathbb{Z} \in \mathscr{B}^{-1}.$
Following Zagier (see page $165$ in \cite{Z}), define a sequence of numbers
$$0<\cdots<A_2<A_1<A_0<A_{-1}<A_{-2}<\cdots$$
by
\begin{eqnarray*}
&& A_k=\frac{1}{w_1w_2...w_k}, \, (k\geq 1), 
\\
&& A_0=1, \\
&& A_{-k}=w_0 w_{-1}\cdots w_{-k+1} \, (k\geq 1) .
\end{eqnarray*}
It is known  that
any number $\lambda\in \underline{b}$ can be written (for each $k\in \mathbb{Z}$)  in the form  \cite[p.~165]{Z}
\begin{eqnarray}\label{lam}
\lambda=pA_{k-1}+qA_k =A_k(pw_k+q) 
\end{eqnarray}
In fact (see \cite{Z}),  ,  there is  a   one to one correspondence
$$\{\lambda\in \underline{b} : \lambda\gg 0\} \longleftrightarrow \{(k, p, q) : k, p, q\in \mathbb{Z}, 
p\geq 1, q\geq 0\}.$$
Now define a linear map $\mathcal{L}$, using the above correspondence, as 
\begin{eqnarray*}
\mathcal{L}(\lambda)=(p, q) \in \mathbb{Z}_{\geq 0}^2.
\end{eqnarray*}
Furthermore,  let the action of the fundamental unit 
$\epsilon$ of $K  (\epsilon > 1, N(\epsilon)=1)$ 
with respect to the basis  $\{w_k, 1\}$ be  given by some matrix 
$M=\sm a&b\\c&d\esm \in  \mathrm{SL}(2, \mathbb{Z}):$
\begin{eqnarray*}
&& \epsilon w_k=a w_k+b  \nonumber\\
&& \,\epsilon =cw_k+d.
\end{eqnarray*}
Under this action, we get 
\begin{eqnarray*}
 \mathcal{L}(\lambda)=(p, q) \rightarrow
\mathcal{L}(\epsilon \lambda)=(p, q)\sm a&b\\ c& d \esm.
\end{eqnarray*}
Since $N(\lambda)=N(\epsilon \lambda),$
we see that
\begin{eqnarray*}
&& \displaystyle \frac{e^{2\pi i   < \mathcal{L}(\lambda),(\alpha, \beta)   > }}{N(\lambda)}
=
\frac{e^{2\pi i   <\mathcal{L}(\epsilon\lambda), (\alpha, \beta) > }}{N(\epsilon\lambda)} 
  \\
 \mbox{ iff \, \,\,} && 
   \sm a-1&b\\c & d-1 \esm  \sm \alpha\\ \beta \esm \in \mathbb{Z}^2 \,\,\,
  \mbox{iff \,\,\,} \sm \alpha\\ \beta\esm \in \mathcal{S},
 \end{eqnarray*}
where the set $\mathcal{S}$ is given by
\begin{eqnarray}\label{set A}   
 \mathcal{S}:=\{(\alpha,\beta)\in\mathbb{Q}^2:N(\epsilon-1)(\alpha,\beta)\in\mathbb{Z}^2 \} 
\end{eqnarray}
and  $<,>$ denotes the usual inner product.
The condition $(\alpha,\beta)\in\mathcal{S}$ follows from the fact that $det\sm a-1& b\\ c& d-1\esm=2-(a+d)=N(\epsilon-1).$
\medskip

\begin{example}
Let $K=\mathbb{Q}(\sqrt{3}). $   Note   that    $U_+=<\epsilon=2 +\sqrt{3}>$   is    the group of positive units. 
Take $w_1=\frac{3+\sqrt{3}}{2}, w_2=\frac{3+\sqrt{3}}{3}$ so that $A_0=1, A_1=\frac{3-\sqrt{3}}{3},
A_2=2-\sqrt{3}.$ There is a one to one correspondence between  
\begin{eqnarray*}
 \{\lambda_2=p_2A_1+q_2A_2, \lambda_1=p_1A_0+q_1A_1\} \longleftrightarrow \{(2, p_2, q_2), (1, p_1, q_1)\}. 
\end{eqnarray*}
   Since    $N(\lambda_1)=N(\epsilon \lambda_1),  $   we    take $(\alpha, \beta)$ such that
$ N(\epsilon -1)(\alpha, \beta) =-2(\alpha, \beta)\in \mathbb{Z}^2 . $ 
 \end{example}

Now let us define the following zeta function associated with real quadratic fields:
 \begin{definition}   Let $U_+$ be a group of totally positive units.  
For $(\alpha,\beta)\in\mathcal{S}$, define 
\begin{eqnarray}\label{secondzeta}
\zeta(s, (\alpha, \beta), \mathscr{B})=N(\underline{b})^s
\sum_{\substack{\lambda \in \underline{b} /U_+\\ \lambda\gg0}}\frac{e^{2\pi i
<\mathcal{L}(\lambda), (\alpha, \beta)>}}{|N(\lambda)|^s}.
 \end{eqnarray}
\end{definition}
%Note that this zeta function is analogous to its counterpart for imaginary quadratic fields, as defined in \eqref{qzeta1}.

Observe that the zeta function $\zeta(s, (\alpha, \beta), \mathscr{B})$ in (\ref{secondzeta}) is well defined from the choice of $(\alpha, \beta) \in \mathcal{S}.$

For each  $\lambda=A_k(pw_k+q)$ in (\ref{lam}), let
\begin{align*}
 N(\lambda)=(w_0-w_0')Q_k(p,q), \, \, 
\end{align*} 
and 
\begin{align}
Q_k(x,y)=\frac{1}{w_k-w_k'}(y+xw_k)(y+xw_k').\nonumber
\end{align}
Note that $Q_j(x,y)$ is an indefinite binary quadratic form   with   positive real coefficients and discriminant $1$ with the property that $0<w_j'<w_j$ are roots of the quadratic equation $Q_j(1,-x)=0$. With this notation, there is    a   one to one correspondence \cite{Z}
\begin{align}\label{corres}
\mathrm{Red}(\mathscr{B})=\{w_1,w_2,\cdots,w_r\} \longleftrightarrow \{Q_1,Q_2,\cdots,Q_r\}.
\end{align}
Now the zeta function in (\ref{secondzeta}) becomes \cite[p.~166]{Z}
\begin{eqnarray}\label{qdr}
  D^{\frac{s}{2}}\zeta(s, (\alpha, \beta), \mathscr{B})
=
\sum_{j=1}^r\sum_{p\geq 1,q\geq0}
\frac{e^{2\pi i (p \alpha+q\beta)}}{Q_j(p,q)^s}.
\end{eqnarray}

%%\begin{remark}
%%In the special case where $(\alpha, \beta)\in \mathbb{Z}^2$, the zeta function in \eqref{qdr} coincides with that studied by Zagier in \cite{Z}, where the first Kronecker  limit formula for real quadratic fields is derived. Furthermore, Vlasenko and Zagier  \cite[Theorem 2]{VZ} provided the higher first Kronecker  limit formula for it.
%%\end{remark}

Before stating our main results, it is essential to provide the definition of an important function, which we refer to as the \emph{higher Herglotz-Zagier-Novikov function}. This function plays a key role in our subsequent discussions.
\begin{definition}
Let  $\alpha, \beta \in \mathbb{R}$ and $\beta \notin \mathbb{Z}.$ Then,
for a positive integer $k\geq1, $ we define the \emph{higher Herglotz-Zagier-Novikov function}
$\mathscr{F}_{k}(x;  \alpha,\beta )$ by
\begin{eqnarray}\label{int repres for hh}
\mathscr{F}_{k}(x;  \alpha,\beta ):=\int_0^\infty \frac{ {\mathrm{Li}}_{k-1}\left(e^{-xt}e^{2\pi i\alpha}\right)}{1-e^{-t}e^{2\pi i\beta}}dt,  \qquad  \mathrm{Re}(x) >0,
\end{eqnarray}\\
where $\displaystyle\mathrm{Li}_{s}(z):=\sum_{n\geq 1}\frac{z^n}{n^s},\ (s,z\in\mathbb{C}, |z|<1)$ is the polylogarithm function \cite{lewin}, \cite[p.~611, Equation (25.12.10)]{nist}.
\end{definition}
In Section \ref{our function}, we derive various properties of $\mathscr{F}_{k}(x; \alpha, \beta)$, including its analytic continuation, as well as   two-, three-,   and six-term functional equations, among other properties.

%\medskip

%\begin{proposition}\label{mz}
%Take $\alpha,\beta\in\mathcal{S}$ in (\ref{set A}). 
%For
%$$Q(x,y)=ax^2+bxy+cy^2, a,b,c \in \mathbb{R}^+, b^2-4ac=1$$
%an indefinite binary quadratic form, 
%define
%\begin{eqnarray*}\label{epstein zeta}
%Z_Q (s; \alpha, \beta) =\sum_{p\geq 1, q\geq 0}
%  \frac{
%e^{2\pi i (\alpha p+ \beta q)}}{Q(p,q)^s} ,
% \end{eqnarray*}
%then for the zeta function of a narrow ideal class $\mathscr{B}$ of a real quadratic field of discriminant $D$, 
%we have
%\begin{eqnarray*}
%D^{\frac{s}{2}}\zeta(s, (\alpha, \beta),\mathscr{B})=
%\sum_{k=1}^r Z_{Q_k}(s; \alpha, \beta)
%\end{eqnarray*}
%where $r$ is the length of $\mathscr{B}$, i. e. the number of reduced numbers in $Red(\mathscr{B}) $   and the binary forms $Q_k$ as defined in \eqref{qe}.
% \end{proposition}
%   

\subsection{The second Kronecker limit formula of Zagier's zeta function}

Zagier \cite{Z} derived the first Kronecker limit formula of the zeta function associated   with     certain indefinite forms. More precisely, let us define the zeta function \cite[p.~166]{Z} (also see \cite{dit}), referred to here as \emph{Zagier's zeta function},
\begin{align}\label{zagier zeta function}
Z_Q(s):=\sum_{p\geq1,q\geq0}\frac{1}{Q(p,q)^s}\qquad (\mathrm{Re}(s)>1),
\end{align}
where $Q(p,q)$ is an indefinite binary quadratic form with positive coefficients and  normalized to have discriminant $1$. So the roots $w'<w$ of $Q(1,-x)=0$ are positive. Then Zagier \cite[p.~167, Theorem]{Z} found the following result:
\begin{theorem}[The first Kronecker limit formula of Zagier's zeta function]\label{zagier first limit formula} The function $Z_Q(s)$ has an analytic continuation to the half-plane $\mathrm{Re}(s)>\frac{1}{2}$ with a simple pole at $s=1$ and 
\begin{align*}
\lim_{s\to1}\left(Z_Q(s)-\frac{\frac{1}{2}\log(w/w')}{s-1}\right)=P(w,w'),
\end{align*}
where $P(x, y)$ is a universal function of two variables:
\begin{align}
P(x,y)=F(x)-F(y)+\mathrm{Li}_2\left(\frac{y}{x}\right)-\frac{\pi^2}{6}+\log\left(\frac{x}{y}\right)\left(\gamma-\frac{1}{2}\log(x-y)+\frac{1}{4}\log\left(\frac{x}{y}\right)\right),\nonumber
\end{align}
 and   where    the function $F(x)$ is the Herglotz function defined in \eqref{her int}.
%\begin{align}
%F(x):=\sum_{n=1}^\infty\frac{\psi(nx)-\log(nx)}{n}\quad (x\in\mathbb{C}\backslash(-\infty,0]).
%\end{align}

\end{theorem}

To study the second Kronecker limit formula of Zagier's zeta function, define the function
\begin{align}\label{epstein zeta}
%Z(s,x, y; \alpha, \beta):=\sum_{p\geq 1, q\geq 0}
% e^{2\pi i (\alpha p+ \beta q)}\left(\frac{
%(x-y) }{(px+q)(py+q)} 
%\right)^s.
Z_Q(s; (\alpha, \beta)):=\sum_{p\geq 1, q\geq 0}\frac{e^{2\pi i (\alpha p+ \beta q)}}{Q(p,q)^s}, \quad (\alpha,\beta\in\mathbb{R}\ \mathrm{and\ Re}(s)>1),
\end{align}
where $Q(x,y)$ is an indefinite binary quadratic form with positive real coefficients and discriminant $1$ with the property that $0<w'<w$ are roots of the quadratic equation $Q(1,-x)=0$.

The convergence of the series in \eqref{epstein zeta} in the region Re$(s)>1$ is evident from the absolute convergence of the series in \eqref{zagier zeta function}. 
%Moreover, when $\alpha$ and $\beta$ are integers, the function $Z_Q(s; (\alpha, \beta))$ reduces to the Zagier's zeta function \eqref{zagier zeta function}.
% Moreover, this zeta function can be viewed as a real quadratic field analogue of Kronecker’s classical zeta function for imaginary quadratic fields, defined in \eqref{qzeta1} above.
%\end{remark}

With the above notations,  our next theorem presents the second Kronecker limit formula of the Zagier's zeta function.
\begin{theorem}\label{mz}
Let $\alpha,\beta\in\mathbb{R}\backslash \mathbb{Z}$.  Then,
\begin{enumerate}
\item  The zeta function  $Z_Q(s; (\alpha, \beta))$ in \eqref{epstein zeta} is an analytic function in the half-plane  $\mathrm{Re}(s)>1/2$.\\
\item The second Kronecker limit formula is given by
\begin{align}\label{at s=1}
Z_Q(1; (\alpha, \beta))=\mathscr{F}_{2} (w';\alpha, \beta)-\mathscr{F}_{2} (w; \alpha, \beta),
\end{align}
where the numbers $w'<w$ are  the roots  of the quadratic  equation $Q(1,-x)=0$ and $\mathscr{F}_{2} (x;\alpha, \beta)$ denotes the higher Herglotz-Zagier-Novikov function defined in \eqref{int repres for hh}.
\end{enumerate}
\end{theorem}

\begin{remark}
Observe that our result in \eqref{at s=1} does not involve any limit, similar to the classical second Kronecker limit formula in the imaginary quadratic field case  \eqref{2nd imaginary}. This is because, when $\alpha,\beta\in\mathbb{R}\backslash\mathbb{Z}$, the series defining $Z_Q(s; (\alpha, \beta))$ in \eqref{epstein zeta} is already analytic in the region $\mathrm{Re}(s)>1/2$ and has no poles. Whereas, when $\alpha$ and $\beta$ are integers, analytic continuation is required and the zeta function $Z_Q(s; (\alpha, \beta))=Z_Q(s)$ has a simple pole at $s=1$.
\end{remark}
 
Our next result treats the zeta function $Z_Q(s; (\alpha, \beta))$ for higher natural numbers.

\begin{theorem}\label{main}
Let\footnote{Note that this result remains valid for all $(\alpha,\beta)\in \mathbb{R}$, provided one takes an appropriate limit on the right-hand side of \eqref{general limit formula} when either $\alpha$ or $\beta$, or both, are integers (see Proposition \ref{special cases} below). Moreover, if $\alpha$ and $\beta$ are non-integral real numbers, then \eqref{general limit formula} is valid for every integer $k \geq 1$.} $\alpha,\beta\in \mathbb{R}$ and let
 $\mathscr{F}_{k}(x;  \alpha,\beta )$  denote the higher Herglotz-Zagier-Novikov function as defined in \eqref{int repres for hh}. Then,   for    any positive integer $k>1$, we have
\begin{align}\label{general limit formula}
Z_Q(k;(\alpha,\beta)) =-(\mathscr{D}_{k-1}\mathscr{F}_{2k})(w,w'; \alpha, \beta),
\end{align}
with the differential operator $\mathscr{D}_{n}$ defined as \cite[p.~33, Definition 6]{VZ}\textup{:}
\begin{align}\label{doperator}
(\mathscr{D}_{n}f)(x,y; \alpha, \beta)=\sum_{i=0}^n\binom{2n-i}{n}\frac{f^{(i)}(x;\alpha, \beta)-(-1)^if^{(i)}(y; \alpha, \beta)}{i!(y-x)^{n-i}},
\end{align}
for $ n\in\mathbb{N}\cup\{0\}$ and any function $f$.
 \end{theorem}

For any  $(\alpha,\beta)\in\mathbb{R}^2$, let us define the function
\begin{align}\label{general zeta}
\mathcal{Z}(s,(\alpha,\beta),\mathscr{B}):=\sum_{j=1}^rZ_{Q_j}(s;(\alpha,\beta)),  \quad (\mathrm{Re}(s)>1),
\end{align}
where $\{Q_j|\ 1\leq j\leq r\}$ are from the correspondence  of Red$(\mathscr{B})$ in \eqref{corres}.
%each $Q_j(x,y)$ for $j=1,\cdots,r$ is an indefinite binary quadratic form with positive coefficients and discriminant 1, such that the roots $w_j>w_j'>0$ of the quadratic equation $Q_j(1,-x)=0$ lie in the set Red($\mathscr{B}$) defined in \eqref{redB}.

Using Theorem \ref{main}, we can find that
\begin{align}\label{table eqn}
\mathcal{Z}(k,(\alpha,\beta),\mathscr{B})=-\sum_{w\in\mathrm{Red}(\mathscr{B})}(\mathscr{D}_{k-1}\mathscr{F}_{2k})(w,w'; \alpha, \beta).
\end{align}
 
In the next subsection, we show that the properties of the function $\mathcal{Z}(s,(\alpha,\beta),\mathscr{B})$ yield corresponding formulas for the zeta function over real quadratic fields.

 \subsection{The second Kronecker limit formula for the zeta function over real quadratic fields}
 
Note that when $(\alpha,\beta)$ is restricted to the set $\mathcal{S}$ in \eqref{set A}, the function $\mathcal{Z}(s, (\alpha, \beta ),\mathscr{B})$, defined in \eqref{general zeta}, essentially becomes the zeta function over real quadratic fields introduced in \eqref{qdr}:
\begin{align}\label{connection b/w real and our}
\zeta(s, (\alpha, \beta ),\mathscr{B})=D^{\frac{s}{2}}\mathcal{Z}(s,(\alpha,\beta),\mathscr{B}).
\end{align}
%where each $Q_j(x,y)$ for $j=1,\cdots,r$ is an indefinite binary quadratic form with positive coefficients and discriminant 1, such that the roots $w_j>w_j'>0$ of the quadratic equation $Q_j(1,-x)=0$ lie in the set Red($\mathscr{B}$) defined in \eqref{redB}.

The following result establishes the second Kronecker limit formula and its higher analogue of the zeta function defined over   a    real quadratic field.
 \begin{theorem}\label{real second kronecker limit formula}
Let the set  $\mathrm{Red}(\mathscr{B})$ and  the set $\mathcal{S}$ be defined in \eqref{redB} and \eqref{set A} respectively. Let $\mathscr{F}_{k}(x;  \alpha,\beta )$  be the higher Herglotz-Zagier-Novikov function defined in \eqref{int repres for hh}.
\begin{enumerate}
\item Let $(\alpha,\beta)\in \mathcal{S}\backslash\mathbb{Z}^2$. The zeta function $\zeta(s, (\alpha, \beta ),\mathscr{B})$  is analytic in the region $\mathrm{Re}(s)>1/2$.

 \item Let $(\alpha,\beta)\in \mathcal{S}$. For  any positive integer $k>1,$
\begin{align}\label{limit formula}
 D^{\frac{k}{2}}\zeta( k, (\alpha, \beta ),\mathscr{B})&=-  \sum_{w\in \mathrm{Red}(\mathscr{B})}P_k(w, w'; \alpha, \beta),
%&=-  \sum_{w\in \mathrm{Red}(\mathscr{B})} Z_k (w,w', \alpha, \beta ), 
\end{align} 
where
  $$P_k(x,y;\alpha, \beta)= (\mathscr{D}_{k-1}\mathscr{F}_{2k})(x,y; \alpha, \beta). $$
  \item   In particular, for $(\alpha,\beta)\in \mathcal{S}\backslash\mathbb{Z}^2$, the expression $$D^{ \frac{1}{2}}\zeta( 1, (\alpha, \beta ),\mathscr{B})= \sum_{w\in \mathrm{Red}(\mathscr{B})}\bigl(\mathscr{F}_{2} (w';\alpha, \beta)-\mathscr{F}_{2} (w; \alpha, \beta)\bigr) $$ can be regarded as the {\textit{second Kronecker limit formula}} of a real quadratic field.  It serves as an analogue of the corresponding formula for imaginary quadratic fields, which is given in \eqref{2nd imaginary}. 
  \end{enumerate}
\end{theorem}

\begin{remark}
\begin{enumerate}
\item Note that the result in \eqref{limit formula} above remains valid for all $(\alpha,\beta)\in \mathcal{S}$, provided one takes an appropriate limit on the right-hand side of \eqref{limit formula} when either $\alpha$ or $\beta$, or both, are integers \textup{(}see Proposition \ref{special cases} below\textup{)}.
\item As a special case of the above theorem, one can get the higher Kronecker limit formula of Vlasenko and Zagier \cite[p.~25, Theorem 2]{VZ}.
\end{enumerate}
\end{remark}

We end this section with a note that the values of $(\mathscr{D}_{n}f)(x,y; \alpha, \beta)$ are solely determined by the $(2n+1)$st derivative of the function $f(x; \alpha, \beta)$ by means of the relation \cite[pp.~33-34, Proposition 7(v)]{VZ}, specifically, if $f$ is any continuously differentiable $2n+1$ times between $x$ and $y$ then
\begin{align*}
(\mathscr{D}_{n}f)(x,y;\alpha,\beta)=\frac{1}{(n!)^2}\int_y^x\left(\frac{(x-t)(t-y)}{x-y}\right)^nf^{(2n+1)}(t; \alpha, \beta)dt.
\end{align*}

%%%%%%%%%%%%%%%%%%%%%%%%%%%%%%%%
\section{Higher Herglotz-Zagier-Novikov function}\label{our function}
%%%%%%%%%%%%%%%%%%%%%%%%%%%%%%
%In this section let  $\alpha, \beta \in \mathbb{R}$ and $\beta \notin \mathbb{Z}.$
%\noindent
%For a positive integer $k\geq2, $ we define the \emph{higher Herglotz-Zagier-Novikov function}
%$\mathscr{F}_{k}(x;  \alpha,\beta )$ by
%\begin{eqnarray}\label{int repres for hh}
%\mathscr{F}_{k}(x;  \alpha,\beta ):=\int_0^\infty \frac{ {\mathrm{Li}}_{k-1}\left(e^{-xt}e^{2\pi i\alpha}\right)}{1-e^{-t}e^{2\pi i\beta}}dt,  \qquad  \mathrm{Re}(x) >0,
%\end{eqnarray}\\
%where $\displaystyle\mathrm{Li}_{s}(z):=\sum_{n\geq 1}\frac{z^n}{n^s},\ (s,z\in\mathbb{C}, |z|<1$) is the polylogarithm function \cite{lewin}, \cite[p.~611, Equation (25.12.10)]{nist}.

Note that $\mathscr{F}_{k}(x;\alpha,\beta)$  in \eqref{int repres for hh} can be written as   
\begin{eqnarray}\label{higher herglotz}
\mathscr{F}_{k}(x;\alpha,\beta )=\sum_{p=1}^\infty\frac{e^{2\pi ip\alpha}\psi_\beta(px)}{p^{k-1}},
\end{eqnarray}
 with 
\begin{eqnarray}\label{poly}
\psi_{\beta}(x):=
 \int_0^\infty\frac{e^{-xt}}{1-e^{-t}e^{2\pi i \beta}}dt.  \qquad  
\end{eqnarray}
 Since the function $\psi_{\beta}(x)$ for  $\mathrm{Re}(x)>0$ in \eqref{poly} can be expressed as 
\begin{eqnarray} 
\psi_{\beta}(x)=\sum_{q\geq 0}\frac{e^{2\pi i \beta q}}{ x+q },\nonumber
\end{eqnarray}
equation \eqref{higher herglotz} implies that 
\begin{eqnarray}\label{hf}
\mathscr{F}_{k}(x; \alpha,\beta)
=\sum_{p\geq 1,q\geq0 }\frac{e^{2\pi i (\alpha p+ \beta q)}}{p^{k-1}(px+q)}.
\end{eqnarray}
 The above representation provides the analytic continuation of $\mathscr{F}_{k}(x; \alpha,\beta)$ to the region $x\in\mathbb{C}':=\mathbb{C}\backslash(-\infty,0]$. This is established in Proposition \ref{anly} below.

Before proceeding further, we first make some remarks.
\begin{remark}\label{r}
\begin{enumerate}
\item The inclusion of the term ``higher" in the name of the function $\mathscr{F}_k(x;\alpha,\beta)$ stems from the fact that it encompasses the previously studied case with $k=2$. In our previous work \cite{CK}, we investigated the special case $\mathscr{F}_2(x;\alpha,\beta)$ and labelled it as the ``Herglotz-Zagier-Novikov function". By extending our analysis to the more general form $\mathscr{F}_k(x;\alpha,\beta)$, we are able to explore a broader range of scenarios and derive the second Kronecker and higher ``limit" formulas.
\item  Observe that $\psi_\beta(x)$ is nothing but a special case of the Lerch zeta function $\phi(z,x, s)$, namely,
$ \phi\left( \beta,x , 1 \right)= \psi_\beta(x).$
  The Lerch zeta function is defined as \cite{lerch}
\begin{align}\label{Lerch}
 \qquad\qquad\phi\left(z, x, s\right)   :=\sum_{q=0}^\infty\frac{e^{2\pi iqz}}{(q+x)^s}, \quad \mathrm{Re}(s)>0,\ \mathrm{if}\ z\in\mathbb{R}\backslash\mathbb{Z}, \quad x  \in \mathbb{C}\backslash \mathbb{Z}_{\leq 0}. 
\end{align} 

%Although, we adopt the notation $\psi_\beta(x)$ to align our higher Herglotz-Zagier-Novikov function $\mathscr{F}_{k}(x;\alpha,\beta)$ with other existing Herglotz-type functions in the literature. Apart from \cite{VZ} and \cite{Z}, also see \cite{CK}, \cite{dgk}, \cite{ishibashi} and \cite{masri} for other Herglotz-type functions.

\item The following function appeared in the work of Novikov \cite{novikov} on the first Kronecker limit formula:
\begin{eqnarray}\label{nov}
\rho(x, \alpha, \beta) =\int_0^{1}\frac{\ln{(1-t^x e^{2\pi i \alpha})}}{e^{-2\pi i \beta}-t}  dt, \quad  \mathrm{Re}(x)>0,
\end{eqnarray} 
where $\alpha\in\mathbb{R},\beta\in\mathbb{R}\backslash\mathbb{Z}$. Note that when $k=2$, $\mathscr{F}_{k}(x;\alpha,\beta )$ reduces to the function $\rho( x, \alpha,\beta )$:
\begin{align}\label{above exp}
\mathscr{F}_{2}(x;\alpha,\beta )= -\rho( x, \alpha,\beta )-\frac{1}{x}\mathrm{Li}_2\left(e^{2\pi i\alpha}\right).
\end{align}
In our previous paper \cite{CK}, we extensively studied the properties of $\mathscr{F}_{2}(x;\alpha,\beta )$ or $\rho( x, \alpha,\beta )$.

\item From \eqref{nov} and \eqref{above exp}, it is straightforward to see that
\begin{align}
\quad\qquad\mathscr{F}_{2}\left(x; \frac{1}{2},\frac{1}{2} \right)=\int_0^1\frac{\log\left(1+t^x\right)}{1+t}dt-\frac{1}{x}\mathrm{Li}_2(-1)= J(x)+\frac{\pi^2}{12x}. \nonumber
\end{align} 
The function $J(x):=\int_0^1\frac{\log\left(1+t^x\right)}{1+t}dt$ is studied in \cite{RZ}, where Radchenko and Zagier investigated the connection among $J(x)$, Stark's conjecture,  cohomology of the modular group and so on.  Also see \cite{cohen, mw}. We have further studied the arithmetic properties of $J(x)$ in a recent paper \cite{CK}.

\item In the case of  $\alpha, \beta \in \mathbb{Z},$  $ \mathscr{F}_{2}\left(x; \alpha, \beta \right)$ does not converge. Instead,
the following Herglotz function is introduced in \cite{Z}  to study the first Kronecker  limit formula of real quadratic fields\textup{:} for $ \mathrm{Re}(x) >0,$
%\begin{enumerate}\item  $ \rho(x,0,0 ):= \int_0^1\bigl(\frac{1}{1-t}+\frac{1}{\log{t}} \bigr)  \log(1-t^x) \frac {dt }{t}$\\
%\item
\begin{align*}
 F(x) 
=&\int_0^1\left(\frac{1}{1-t}+\frac{1}{\log{t}} \right)  \log(1-t^x) \frac {dt }{t}\nonumber  \\
=&\sum_{n\geq 1} \frac{\psi(nx)-\log{(nx)}}{n},    
\end{align*}
  where    $\displaystyle\psi(x) :=  \frac{\Gamma'(x)}{\Gamma(x)}=
\int_0^{\infty}
\left(\frac{e^{-t}}{t}-\frac{e^{-xt}}{1-e^{-t}}\right)dt$ is the usual digamma function. However, $ \mathscr{F}_{2}\left(x; \alpha, \beta \right)$ reduces to $F(x)$ in the limiting case up to correction factors. See Proposition \textup{\ref{special cases}} below.
\item For each $k>2,$  the following higher Herglotz function was introduced in \cite{VZ} to study the higher Kronecker ``limit" formula of real quadratic fields:
\begin{eqnarray}\label{hk}
\mathscr{F}_k(x)  =\sum_{n\geq 1} \frac{\psi(nx) }{n^{k-1}}.
\end{eqnarray}  
It is shown in Proposition \textup{\ref{special cases}} that $ \mathscr{F}_{k}\left(x\right)$ is also a limiting case of $\mathscr{F}_{k}\left(x;\alpha, \beta \right)$.
%\item One can take more general parameters  $\alpha, \beta $ in $\mathscr{F}_{k}(x; \alpha,\beta)$   so that its definition can be extended  \textup{(}see \cite{CK}\textup{):}
%$$\int_0^\infty\frac{\mathrm{Li}_{k-1}\left(e^{-xt}u\right)}{e^tv^{-1}-1}dt,\   k \geq2, \ u\in\mathbb{C}\backslash\{0\cup(1,\infty)\}, v\in\mathbb{C}\backslash\{0\cup[1,\infty)\}.$$ 
\end{enumerate}
%\end{enumerate}
\end{remark}

The functions $F(x)$ and $\mathscr{F}_k(x)$ can be obtained as limiting cases of our function $\mathscr{F}_k(x;\alpha,\beta)$:
\begin{proposition}\label{special cases}
Let the functions $F(x)$ and $\mathscr{F}_k(x)$ be defined in \eqref{her int} and \eqref{hk}. Let $\gamma$ be Euler's constant. For $x\in\mathbb{C}'$, we have
\begin{enumerate}
\item $\displaystyle \lim_{\substack{\alpha,\beta\to 0}}\bigg\{-\mathscr{F}_2(x;\alpha,\beta)+\left(\gamma+\log(x)+\log(1-e^{2\pi i \beta})\right)\log\left(1-e^{2\pi i \alpha}\right)+\mathrm{Li}_s'(e^{2\pi i \alpha})|_{s=1}\bigg\}\\=F(x).$\\
\item $\displaystyle \lim_{\alpha,\beta\to0}\left\{-\mathscr{F}_k(x;\alpha,\beta)-\left(\gamma+\log(1-e^{2\pi i\beta}\right)\mathrm{Li}_{k-1}\left(e^{2\pi i\alpha}\right)\right\}=\mathscr{F}_k(x),\ k>2, k\in \mathbb{N}.$
\item $\displaystyle \lim_{\alpha,\beta\to0}\left\{-\mathscr{F}_1(x;\alpha,\beta)-\frac{\left(\gamma+\log(1-e^{2\pi i\beta})+\log(x)\right)}{e^{-2\pi i\alpha}-1}+\frac{\log(1-e^{2\pi i\alpha})}{2x}+\mathrm{Li}_s'(e^{2\pi i \alpha})|_{s=0}\right\}\\=\mathscr{F}_1(x)$,
\end{enumerate}
 where $\mathscr{F}_1(x)$ is the Ramanujan period function which is defined as \cite[Equation (1.2)]{ck2}
\begin{align*}
\mathscr{F}_1(x)=\sum_{p=1}^\infty\left(\psi(px)-\log(px)+\frac{1}{2px}\right).
\end{align*}
In a recent work \cite{ck2}, the present authors studied the function $\mathscr{F}_1(x)$ in the context of period functions of Maass cusp forms.
\end{proposition}

\medskip
%%%%%%%%%%%%%%%%%%%%%%%%%%
\subsection{Asymptotics and analytic continuation}
%%%%%%%%%%%%%%%%%%%%%%%%%%%

For a function $f$, let us denote an asymptotic expansion $\displaystyle f(x) \sim \sum_{n=n_0}^{\infty} a_n(x-a)^n$ as $x\rightarrow a$  in the sense  
 that 
$$f(x)-\sum_{n=n_0}^{N-1}a_n(x-a)^n=\mathcal{O}\left((x-a)^N\right)$$  when $x\rightarrow a$ for every $N.$  When $a=\infty ,$ replace $x-a$ by $\frac{1}{x}.$ With this notation, we have the following:

\begin{proposition}\label{ay}
Let $\alpha, \beta \in\mathbb{R}$ and $\beta\notin\mathbb{Z}.$
\begin{enumerate}
\item  As $x\rightarrow \infty$ in the sector $|\arg(x)|\leq \pi-\delta$, $\delta>0$,

\begin{align}\label{asymptotic of psi_beta}
\psi_\beta(x)\sim\sum_{n=0}^\infty\frac{a_n(\beta)}{x^{n+1}}, \, 
\end{align}
and

\begin{align}\label{asymp of hnf large x}
\mathscr{F}_{k}(x; \alpha,\beta)\sim\sum_{n=0}^\infty\frac{a_n(\beta)}{x^{n+1}}\mathrm{Li}_{k+n}\left(e^{2\pi i\alpha}\right),
\end{align}
where
\begin{align*}
 a_n(\beta):=\frac{d^{n}}{dt^{n}}\left(\frac{1}{1-e^{-t}e^{2\pi i \beta}}\right)\bigg|_{t=0}.
 \end{align*}

%\end{eqnarray}
%\end{enumerate}

\item  As   $x\to 0$ in the sector $|\arg(x)|\leq\pi-\delta$, $\delta>0$, 
\begin{align}\label{asymp of hnf small x}
\quad\mathscr{F}_{k}(x;\alpha,\beta)&\sim\frac{1}{x}\mathrm{Li}_{k}\left(e^{2\pi i\alpha}\right)+\sum_{r=1}^{k-1}(-x)^{r-1} \mathrm{Li}_{k-r}\left(e^{2\pi i\alpha}\right) \mathrm{Li}_{r}\left(e^{2\pi i\beta}\right)\nonumber\\
&\quad-\frac{(-x)^{k-1}}{x+1}\mathrm{Li}_{k-r}\left(e^{2\pi i(\alpha+\beta)}\right)+(-1)^{k-1}\sum_{n=0}^\infty a_n(\alpha) \mathrm{Li}_{k+n}\left(e^{2\pi i\beta}\right)x^{k+n-1}.
\end{align}
\end{enumerate}
\end{proposition}
 
%\bigskip
%\subsection{Analytic Continuation}

The above   asymptotic expansion  in Proposition \ref{ay} along with \eqref{higher herglotz} readily gives the following result. 
\begin{proposition}\label{anly}
Let $\alpha, \beta\in \mathbb{R}, \beta \notin \mathbb{Z}$. The higher Herglotz-Zagier-Novikov function $\mathscr{F}_k(x;\alpha,\beta),\ k\geq2$ can be extended analytically to $x \in \mathbb{C}'.$
\end{proposition}

\medskip

%%%%%%%%%%%%%%%%%%%%%%%%%%%%%%%%%
\subsection{Two-, three-  and six-term functional equations for $\mathscr{F}_{k}(x;\alpha,\beta)$}
%%%%%%%%%%%%%%%%%%%%%%%%%%%%%
%Vlasenko and Zagier \cite{VZ} showed that the  Herglotz function $F(x)$ in (\ref{h2}) and the higher Herglotz function $\mathscr{F}_k(x, 0, 0)$ in (\ref{hk})  satisfies the following two term and three term relations.\\

The following  shows that cohomological relations  hold  for the
higher Herglotz-Zagier-Novikov function $\mathscr{F}_{k}(x ; \alpha,\beta)$: 
 
\begin{theorem}\label{fe thm}
Let $k\in\mathbb{N}$ and $\alpha, \beta\in\mathbb{R}\backslash\mathbb{Z}$. For any $x\in \mathbb{C}'$, we have
\begin{enumerate}
\item
\begin{align} \label{two term}
& \mathscr{F}_{k}(x ; \alpha,\beta)+(-x)^{k-2}\mathscr{F}_{k}\left(\frac{1}{x} ; \beta, \alpha \right)\nonumber\\
&=\frac{1}{x}\mathrm{Li}_k\left(e^{2\pi i\alpha}\right)-(-x)^{k-1}\mathrm{Li}_k\left(e^{2\pi i\beta}\right) 
 +\sum_{r=1}^{k-1}(-x)^{r-1} \mathrm{Li}_{k-r}\left(e^{2\pi i\alpha}\right) \mathrm{Li}_{r}\left(e^{2\pi i\beta}\right).
\end{align}

\item
\begin{align}\label{three term}
 &\mathscr{F}_{k}(x ;\alpha,\beta) 
-\mathscr{F}_{k}(x+1;\alpha+\beta,\beta)
+(-x)^{k-2}\mathscr{F}_{k}\left(\frac{x+1}{x}; \alpha+\beta,\alpha \right) 
\nonumber \\
 &=\frac{1}{x} \mathrm{Li}_k\left(e^{2\pi i\alpha}\right)
-\frac{(-x)^{k-1}}{x+1}\mathrm{Li}_{k}\left(e^{2\pi i(\alpha+\beta)}\right)
+\sum_{r=1}^{k-1}(-x)^{r-1}
\mathrm{Li}_{r,k-r}\left(  e^{2\pi i \beta},e^{2\pi i \alpha}\right).
\end{align}

\item
\begin{align}\label{six term}
 &\mathscr{F}_{k}(x ;\alpha,\beta) + \mathscr{F}_{k}(x ; -\alpha, -\beta)
-\mathscr{F}_{k}(x+1;\alpha+\beta,\beta)-\mathscr{F}_{k}(x+1;-\alpha-\beta,-\beta)\nonumber\\
&
+(-x)^{k-2}\mathscr{F}_{k}\left(\frac{x+1}{x}; \alpha+\beta,\alpha \right)+(-x)^{k-2}\mathscr{F}_{k}\left(\frac{x+1}{x}; -\alpha-\beta,-\alpha \right)
\nonumber\\
 &=\frac{1}{x}\left\{\mathrm{Li}_k\left(e^{2\pi i\alpha}\right) +\mathrm{Li}_k\left(e^{-2\pi i\alpha}\right)\right\}
-\frac{(-x)^{k-1}}{x+1} \left\{\mathrm{Li}_{k}\left(e^{2\pi i(\alpha+\beta)}\right)+\mathrm{Li}_{k}\left(e^{-2\pi i(\alpha+\beta)}\right)\right\}\nonumber\\
&\qquad +\sum_{r=1}^{k-1}(-x)^{r-1}\left\{
\mathrm{Li}_{r,k-r}\left(  e^{2\pi i \beta},e^{2\pi i \alpha}\right) +
\mathrm{Li}_{r,k-r}\left(  e^{-2\pi i \beta},e^{-2\pi i \alpha}\right)\right\},
\end{align}
where $\mathrm{Li}_{a,b}(z_1,z_2)$ is the double polylogarithm, defined in \cite{ggl}, as
\begin{align}\label{double polylog}
\mathrm{Li}_{a,b}( z_1, z_2 ): =\sum_{0<p<q}
\frac{z_1^p z_2^q}{p^aq^b}\ (z_i\in\mathbb{C},\ |z_i|<1).
\end{align}
\end{enumerate}
\end{theorem}
For an explanation of the left-hand sides of these functional equations, see Section \ref{com as} below. The three-term functional equations of the form \eqref{three term} appear in diverse areas of mathematics, including period functions for Maass forms, cotangent functions, and double zeta functions, among others. For a comprehensive discussion on this topic, we refer to Zagier's \cite{zagier} insightful paper.
  \medskip

%%%%%%%%%%%%%%%%%%%%%%%%%%%
\section{Cohomological aspects}\label{com as}
%%%%%%%%%%%%%%%%%%%%%%%%%%%%
Note that the values of $\left(\mathscr{D}_{k-1}\mathscr{F}_{2k}\right)(w,w';\alpha,\beta)$ in the second Kronecker limit formula depend only on the   $(2k-1)$st derivative of the function $\mathscr{F}_{2k}(k;\alpha,\beta)$. In this section, we construct a cocycle class using this derivative (see Proposition \ref{co} below).

Let $M=\begin{pmatrix}
a&b\\
c&d
\end{pmatrix}\in\mathrm{GL}^+_2(\mathbb{R})$. Define an action of $GL^+_2(\mathbb{R}) $ on  the space of functions  $f: \mathbb{C}' \times (\mathbb{R} / \mathbb{Z})^2 \rightarrow \mathbb{C}$ as
\begin{eqnarray*}
 %\left( \large {f\big{|}} \sm a& b\\ c& d\esm\right )(x, (\alpha, \beta))  
 \left( \large {f\big{| } }M\right)(x, (\alpha, \beta))=  f\left(\frac{ax+b}{cx+d}, \left(\alpha, \beta\right)M^t\right), 
\end{eqnarray*}
where $M^t$ denotes the transpose of $M.$

%One can check that$$ \large (f|_{k}M |_k N \large)(x, (\alpha, \beta))=\large (f|_kMN\large )(x, (\alpha, \beta)).$$

The modular group $\Gamma:=\mathrm{SL}_2(\mathbb{Z})$ can be defined with generators $S=\begin{pmatrix}0 & -1\\ 1 & 0\end{pmatrix}$ and $U=\begin{pmatrix}1& -1\\ 1 & 0\end{pmatrix}$, along with the relations:
$$\Gamma =\large< S, U\,  : \,  S^4=U^6= I\large >.$$
Furthermore, let
$$T=\begin{pmatrix}1 & 1\\ 0 & 1\end{pmatrix}=-US,\quad T'=\begin{pmatrix}1 & 0\\ 1 & 1\end{pmatrix}  = -U^2S=TST.$$
 
Let $\mathbb{Z}[\Gamma]$ be the group ring of $\Gamma$ over   $\mathbb{Z}$ and let
 $V$ be a right $\mathbb{Z}[\Gamma]$-module.  
 The cohomology group is
$$H^1(\Gamma, V)=Z^1(\Gamma,V)/ B^1(\Gamma, V)$$ where
$$Z^1(\Gamma,V)=\{ \phi :\Gamma\rightarrow V :\, \phi_{\gamma_1\gamma_2}=\phi_{\gamma_1}|\gamma_2, \forall \gamma\in \Gamma\},$$
and
$$B^1(\Gamma, V):=\{\phi :\Gamma\rightarrow V :\,  \phi_{\gamma}=v_0|\gamma-v_0, \mbox{for some $v_0\in V$}\}.$$
 
%%%%%%%%%%%%%%%%%%%%%%%%%%%%%%%%
\subsection{Cocycle} 
  Now take $V$ as the space of functions on $\mathbb{C}'\times \mathbb{R}^2$ with the action on $\Gamma$ in weight $ 2k$ given by 
$$ F|m:= (F|_{ 2k}m)(x, \alpha, \beta):=
(cx+d)^{-2k}F\left(\frac{ax+b}{cx+d}, (\alpha, \beta) m^t\right),$$ 
for $ m=\sm a& b\\ c& d\esm \in \Gamma.$ 
Then a $1$-cocycle $\phi$ is the same as a pair of functions 
$F=\phi_S$ and    $G=\phi_U=F(x)+(\phi_T|_{2k}S )(x)$   satisfying  
\begin{eqnarray}\label{period}
& F|_{2k}\Large{(}I+(-I)+S+(-S) \Large{)}=0 \nonumber \\
& G|_{2k}\Large{(} I+(-I)+U+(- U)+U^2+(-U^2) \Large{)} =0.
\end{eqnarray}
  So we may rewrite the relation of  $G$ in  (\ref{period}) as
\begin{eqnarray*}
 (F|_{2k} I+(-I))- (F|_{2k}T+(-T)) - (F|_{2k}TST+(-TST))=B , 
\end{eqnarray*}
where
$$B(x; \alpha, \beta):=
 \phi_T|_{2k}S|_{2k}\large{(}I+(-I)+U+(-U)+U^2+(-U^2)  \Large{)}|_{2k}(-S).$$
 
%%%%%%%%%%%%%%%%%%%%%%%%%%%%%%

Now take  $(\alpha, \beta) \in \mathcal{S}$ in (\ref{set A})  and consider 
$$\psi_{\alpha,\beta, k}(x):= sgn(x)\sideset{}{'}\sum_{\tiny{
\begin{array}{cc} p , q\geq 0 \\ (p,q)\neq (0,0)\end{array}
}} \frac{e^{2\pi i (\alpha p+ \beta q)}}{(p|x|+q)^{k}},$$
where the summation is marked with a prime sign, indicating that the boundary terms ($p=0$ or $q=0$) are included with a factor of $1/2$.  Note  that,  for any $x\in \mathbb{C}',$
$$-\frac{1}{(2k-1)!} \left(\frac{d}{dx}\right)^{2k-1}\left(\mathscr{F}_{2k}(x; \alpha,\beta)
  +  \frac{  1}{2x}\mathrm{Li}_{2k}(e^{2\pi i  \alpha}) +  \frac{x^{2k-1}}{2}\mathrm{Li}_{2k}(e^{2\pi i  \beta}) \right) 
 =\psi_{\alpha,\beta, 2k}(x).$$

\begin{proposition} \label{co}
For every $k\geq 1, $ the map 
$$\phi_S=\psi_{\alpha,\beta,2k},\, \,  \phi_T=0 $$
can be uniquely extended to a $1$-cocycle for $\Gamma= \mathrm{SL}_2(\mathbb{Z})$ with coefficients in the space of functions on $\mathbb{C}' \times  \mathbb{R}^2$ with weight $2k.$
\end{proposition}

%%%%%%%%%%%%%%%%5
%\subsection{Binet type integal}
\subsection{ Connection between $\mathscr{F}_{k}(x; \alpha,\beta)$ and a generalized Dedekind eta function} 
%%%%%%%%%%%%%%%%%%%%%%
\bigskip

Consider the following Binet-type integral for $\psi_\beta(x)$ :

\begin{lemma}\label{new int repr}
Let $\psi_\beta(x)$ be defined in \eqref{poly}. Let $\psi(x)$ be the usual digamma function. Let $x>0$ and $0<\beta<1$. Then we have
\begin{align*}
\psi_\beta(x)&=\frac{1}{2x}-\frac{i}{2\pi x}\left(\psi(\beta)-\psi(1-\beta)\right)+i\int_{-\infty}^\infty\left(\frac{1}{it-x}+\frac{1}{ x}\right)\frac{e^{2\pi\beta t}}{e^{2\pi t}-1}dt\\
&=
\frac{1}{2x}-\frac{i}{2\pi x}\left(\psi(\beta)-\psi(1-\beta)\right)\\
&\quad+  \int_{0}^{i \infty} \left(\frac{1}{  t+ x}-\frac{1}{ x}\right)\frac{e^{ 2\pi i \beta t}}{1-e^{ 2\pi it}}dt 
  + \int_0^{i\infty}\left(\frac{1}{ t- x}+\frac{1}{ x}\right)
\frac{e^{  2\pi (1- \beta) it}}{1-e^{  2\pi i t} }dt.
\end{align*}
\end{lemma}

\bigskip
With the aid of the Binet-type integral, given in Lemma \ref{new int repr}, we establish a connection between the higher Herglotz-Zagier-Novikov function $\mathscr{F}_{k}(x; \alpha,\beta)$ and the generalized Dedekind eta-function introduced by Berndt \cite{be}.

\begin{proposition}\label{bi}
Let $k\in\mathbb{N},\ \alpha\in\mathbb{R}$ and  $0<\beta<1$. For $x>0$, we have
\begin{align}\label{cc}
\mathscr{F}_{k}(x; \alpha,\beta) + \mathscr{F}_{k}(x; \alpha, -\beta)=\frac{1}{x}\mathrm{Li}_{k}(e^{2\pi i \alpha})  + \int_{0}^{i \infty} \left(\frac{1}{  t+ x}+\frac{1}{t-x}\right)   H( t,2-k,\alpha,\beta)dt, 
\end{align}
where
\begin{eqnarray*}\label{defnG}
  H( \tau,k,\alpha, \beta):=   A(\tau, k, \beta, \alpha)+A(\tau, k, -\beta,\alpha),  
\end{eqnarray*}
and
\begin{eqnarray}\label{gen ded}
A(\tau, s, \beta, \alpha ):=\sum_{m> -\beta,}\sum_{n \geq 1} n^{s-1} e^{2\pi i n  ( \alpha+\beta \tau)} e^{2\pi i mn\tau },\   \mathrm{Im}(\tau)>0, s\in\mathbb{C},  
\end{eqnarray}
	is the generalized Dedekind eta-function \textup{(}see \cite[p.~496]{be} and \cite{lew}\textup{)}.
\end{proposition}

\bigskip

\begin{remark}
\begin{enumerate}
\item The following generalized Eisenstein series was studied by Berndt \cite{be}:
\begin{align}\label{berndt es}
G(  \tau, s,    \beta,\alpha  ):= \sideset{}{'}\sum_{m,n=-\infty}^\infty\frac{1}{((m+\beta)\tau+n+\alpha)^s},\quad \alpha, \beta\in \mathbb{R},
\end{align}
where  $m$ and $n$ range over all pairs of integers except for the possible pair $m=-\beta$ and $n=-\alpha$.
\begin{enumerate}
\item For $\mathrm{Re}(s)>2,$
\begin{align}\label{gde and}
%&\frac{\Gamma(s)}{(-2\pi i )^s}\left(
%E( \tau,s,  \beta,\alpha )-\sum_{n\geq 0} \frac{1}{(n +\alpha)^s}-e^{2\pi i \frac{s}{2}} \sum_{n\geq 0} \frac{1}{(n-\alpha)^s}\right)\nonumber\\
&\frac{\Gamma(s)}{(-2\pi i )^s} G( \tau,s,  \beta,\alpha )=A(  \tau,  s, \beta, \alpha )+e^{ \pi is} A(  \tau, s,   -\beta, -\alpha).
\end{align} 
\item $G( \tau, s,  \beta,\alpha )$ can be analytically continued to the entire $s$-plane \cite[Theorem 1]{be}.\\

\item The generalized eta function $A(  \tau,  s, \beta, \alpha )$ reduces to the Dedekind eta function $\eta(\tau)$ \cite[p.~499]{be}: $$A(\tau,0,0,0)=\pi i\tau/12-\log\eta(\tau).$$
\end{enumerate}
\item Siegel \cite[p.~47]{Se} derived the functional equation of  the following zeta series
\begin{eqnarray*}
 G^*(\tau, s, \beta, \alpha):= {\sum_{m,n\in \mathbb{Z}}}'\frac{\mathrm{Im}(\tau)^s}{ |(m+\beta)\tau+n+\alpha|^{2s}} ,\quad \mathrm{Re}(s)>2, 
\end{eqnarray*} 
in studying the second Kronecker limit formula of imaginary quadratic fields. 
\end{enumerate}
\end{remark}
\bigskip
\begin{remark}
Lemma \textup{\ref{new int repr}} is an analogue of the Binet integral representation of $\psi(x)$  \cite[p.~251]{ww}:
\begin{align*}
\psi(x)=\log x-\frac{1}{2x}-\int_0^\infty\frac{1}{e^{2\pi t}-1}\frac{2t}{t^2+x^2}dt.
\end{align*}
Using the above integral formula, Radchenko and Zagier obtained the following integral representation of the Herglotz-Zagier function $F(x)$ \cite[Section 7.3]{RZ}:
\begin{align}\label{rdint}
F(x)=-\frac{\pi^2}{12x}-\int_0^{i\infty}\left(\frac{1}{  t+ x}+\frac{1}{t-x}\right)H(t)dt,
\end{align}
where $H(\tau)=\log(q^{1/24}/\eta(\tau))$, here, as usual, $q=e^{2\pi i t}$ and $\eta(\tau)=q^{\frac{1}{24}}\prod_{n\geq 1}(1-q^n), q=e^{2\pi i \tau},$ 
is the classical Dedekind eta function. 
Therefore, \eqref{cc} can be considered as an analogue of \eqref{rdint}.
\end{remark}
%%%%%%%%%%%%%%%%%%%%%%%%%
\section{Numerical examples of the second Kronecker limit formula}\label{num}
In this section, we provide a numerical example of the second Kronecker limit formula in Theorem \ref{main} above.

%Let us define the function
%\begin{align}\label{general zeta}
%\mathcal{Z}(s,(\alpha,\beta),\mathscr{B}):=D^{-s/2}\sum_{j=1}^rZ_{Q_j}(s;(\alpha,\beta)),\ (\alpha,\beta\in\mathbb{R}\backslash\mathbb{Z}),
%\end{align}
%where each $Q_j(x,y)$ for $j=1,\cdots,r$ is an indefinite binary quadratic form with positive coefficients and discriminant 1, such that the roots $w_j>w_j'>0$ of the quadratic equation $Q_j(1,-x)=0$ lie in the set Red($\mathscr{B}$) defined in \eqref{redB}.

Recall that when $(\alpha,\beta)$ is restricted to the set $\mathcal{S}$, the function $\mathcal{Z}(k,(\alpha,\beta),\mathscr{B})$ defined above in \eqref{general zeta} coincides with the zeta function over real quadratic fields introduced in \eqref{qdr}. For instance, in the case $(\alpha,\beta)=\left(0.5,0.5\right)$, which appears in the table below, we have $$D\  \mathcal{Z}(2,(\alpha,\beta),\mathscr{B})=\zeta(2,(\alpha,\beta),\mathscr{B}).$$ 
%Moreover, using Theorem \ref{main}, we can find that
%\begin{align}\label{table eqn}
%D^{k/2}\mathcal{Z}(k,(\alpha,\beta),\mathscr{B})=-\sum_{w\in\mathscr{B}}(\mathscr{D}_{k-1}\mathscr{F}_{2k})(w,w'; \alpha, \beta).
%\end{align}

To prepare a table that shows that both sides of the result in \eqref{table eqn} matches perfectly, we consider the following example.

The number field $K=\mathbb{Q}(\sqrt{3})$ has two narrow classes $\mathscr{B}_0$ and $\mathscr{B}_1$ with the corresponding sets of reduced quadratic irrationalities being \cite[p.~36]{VZ}
\begin{align*}
\mathrm{Red}(\mathscr{B}_0)=\{2+\sqrt{3}\} \qquad \mathrm{and}\qquad \mathrm{Red}(\mathscr{B}_1)=\left\{1+\frac{1}{\sqrt{3}},\frac{3+\sqrt{3}}{2}\right\}.
\end{align*}

We utilized Mathematica to compute the values presented in the following table.
 
%To evaluate $ \mathcal{Z}(2,(\alpha,\beta),\mathscr{B})$ in the third column in the table below, we made use of \eqref{qdr}.
\bigskip

%\textbf{Table: Numerical verification of our higher Kronecker limit formula  \eqref{limit formula}}
\begin{center}
\begin{tabular}{ |p{1.8cm}|p{1.2cm}|p{1.3cm}|p{4.1cm}|p{4.1cm}|}
 \hline
% \multicolumn{2}{|c|}{\textbf{Special values of $T(x)$}} \\
 %\hline
 \vspace{0.1mm}$\mathrm{Red}(\mathscr{B})$  &  \vspace{1mm}\hspace{4mm}$\alpha$ & \vspace{0.1mm}\hspace{4mm}$\beta$& \vspace{0.001mm} \hspace{6mm}$D\ \mathcal{Z}(2,(\alpha,\beta),\mathscr{B})$ & \vspace{0.01mm}\hspace{2mm} RHS of HKLF \eqref{table eqn}\\
 \hline
 \vspace{1.1cm}\hspace{2mm}$\mathrm{Red}(\mathscr{B}_0)$\newline$=\{2+\sqrt{3}\}$   & 0.5 \newline\newline \newline 0.3562 \newline\newline \newline  2.9748 & 0.5\newline \newline \newline -0.4052 \newline\newline \newline 0.6723 & $-11.127412237247254\cdots$\newline$+6.11342\cdots\times 10^{-20 }i$\newline\newline $-7.259415409209688\cdots$\newline $+8.700347577135265\cdots i$\newline\newline $12.451416961455129\cdots-$\newline $2.501571359672215\cdots i$&$-11.12741223912468\dots$\newline $+1.30095\dots\times10^{-15} i$\newline \newline $ -7.259415410306584\dots$\newline$+8.700347578594402\dots i$ \newline\newline $12.451416963412164\dots-$\newline $2.5015713592878965\dots i$\\ 
\hline
\vspace{0.8cm}\hspace{2mm}$\mathrm{Red}(\mathscr{B}_1)$\newline$=\{1+\frac{1}{\sqrt{3}},\newline \frac{3+\sqrt{3}}{2}\}$   & 0.5 \newline\newline \newline 0.3562 \newline\newline \newline  2.9748 & 0.5\newline \newline \newline -0.4052 \newline\newline \newline 0.6723 & $-3.9608460492249566\dots$\newline$+3.87885\dots\times 10^{-20 }i$\newline\newline $-2.562703367148864\dots$\newline $+3.125265764725667\dots i$\newline\newline $4.508649638403528\dots-$\newline $0.6044254869826673\dots i$ & $-3.960846051402042\dots$\newline $+4.48482\dots\times10^{-16} i$\newline \newline $-2.562703368470003\dots$\newline$+3.125265766429505\dots i$ \newline\newline $4.50864964043679\dots-$\newline $0.6044254870852179\dots i$\\ 
\hline
\end{tabular}
\end{center}
  \section{Rational zeta values}\label{zeta values}
%%%%%%%%%%%%%%%%%%%%%%%%%%%%
%\subsection{Analogue of V-Z Theorem Theorem 9}
The following result had been studied in \cite{VZ} when $\alpha \in \mathbb{Z}:$

\begin{lemma}\label{rz}
Let $\alpha\in\mathbb{Q}\backslash\mathbb{Z}$ such that $N(\epsilon-1)\alpha\in\mathbb{Z}$. Let $ \mathscr{B} \rightarrow \mathscr{I}(\mathscr{B})$ be an invariant of narrow ideal classes defined by
\begin{eqnarray}\label{deffxy}
\mathscr{I}( \mathscr{B})=\sum_{w\in \mathrm{Red}(\mathscr{B})} F(w,w', \alpha, \alpha)
\end{eqnarray}
for some function $F : \mathbb{R}^2\times (\mathbb{R}/\mathbb{Z})^2   \rightarrow \mathbb{C}.$
 Suppose $F$ has the form
\begin{eqnarray*}
F(x,y, \alpha,\alpha)= g(x-1,y-1, 0, \alpha) - g\left(1-\frac{1}{x}, 1-\frac{1}{y}, 0,\alpha\right)
\end{eqnarray*}
for some function $g(x,y,0,\alpha).$  Then
\begin{enumerate}
\item For any narrow ideal class $\mathscr{B}$ we have
\begin{eqnarray*}
\mathscr{I}(\mathscr{B})=\sum_{x\in \mathrm{Red}_w(\mathscr{B})} g(x,x', 0,\alpha)-\sum_{x\in \mathrm{Red}_w(\mathscr{B}^*)}g\left(\frac{1}{x}, \frac{1}{x'}, 0,\alpha\right).
\end{eqnarray*}

\item If $g(x,y,0, \alpha)$ has the form $\displaystyle H_1(x-y, 0,\alpha)+H_2\left(\frac{1}{x}-\frac{1}{y}, 0,\alpha\right)$ for some functions
$H_1, H_2 : \mathbb{R} \times (\mathbb{R}/\mathbb{Z})^2  \rightarrow \mathbb{C},$ then $\mathscr{I}(\mathscr{B})=0$ for all classes $\mathscr{B}.$
\end{enumerate}
 \end{lemma}

 \medskip

We record a corollary of the above lemma in the next result.
\begin{coro}\label{rc}
\begin{eqnarray*}
\mathscr{I}(\mathscr{B})\pm\mathscr{I}(\mathscr{B}^*)
=\left(\sum_{x\in \mathrm{Red}_w(\mathscr{B})}\pm
\sum_{x\in \mathrm{Red}_w(\mathscr{B}^*)}\right)G^{\mp}(x,x', 0,\alpha)
\end{eqnarray*}
with $\displaystyle G^{\mp}(x, y, 0,\alpha)=g(x,y, 0,\alpha)\mp g\left(\frac{1}{x}, \frac{1}{y},0,\alpha\right).$
\end{coro}

\medskip
 
\begin{theorem}\label{v-z}
For every narrow ideal class $\mathscr{B},$ integer $k\geq2$ and  $\alpha\in\mathbb{Q}\backslash\mathbb{Z}$ such that $N(\epsilon-1)\alpha\in\mathbb{Z}$, we have
  \begin{eqnarray*}\label{tm9 vz eqn} 
&& D^{\frac{k}{2}}\left(\zeta (k, (\alpha,\alpha), \mathscr{B}) +(-1)^k\zeta (k, (\alpha, \alpha), \mathscr{B}^*)\right)\\
&& =\left(\sum_{x\in\mathrm{Red}_w(\mathscr{B})}+(-1)^k\sum_{x\in\mathrm{Red}_w(\mathscr{B}^*)}\right)W_k(x,x', \alpha,\alpha),
\end{eqnarray*}
where 
\begin{align*}
W_k(x,y,\alpha,\alpha)&=\mathscr{D}_{k-1}\left(\mathscr{F}_{2k}\left(|x|;0, \alpha\right)+|x|^{2k-2}\mathscr{F}_{2k}\left(\frac{1}{|x|};0, \alpha\right)-\frac{3}{4}\left(\frac{1}{|x|}+|x|^{2k-1}\right)\right.\\
&\left.\quad\times\left(\zeta(2k)+\mathrm{Li}_{2k}(e^{2\pi i   \alpha})\right)+\sum_{r=0}^{ 2k-1 }\left(|x|^{2r-1}+|x|^{2k-2r-1}\right) \zeta(2k-2r) \mathrm{Li}_{2r}(e^{2\pi i\alpha})\right).
\end{align*}
\end{theorem}

\begin{remark}
Note that Theorem \ref{v-z} involves combinations of $\zeta(k,(\alpha,\alpha),\mathscr{B})$. However, the ideal scenario would be to have a result that contains combinations of $\zeta(k,(\alpha,\beta),\mathscr{B})$, without $\mathscr{F}_k$ appearing in the definition of $W_k$. Unfortunately, we have not been able to achieve this yet. In this regard, we conjecture the following:
\begin{conjecture}
The expression $\zeta(k,(\alpha,\beta),\mathscr{B})+(-1)^k\zeta(k,(\beta,\alpha),\mathscr{B}^*)$ is a linear combination of polylogarithms of the form $\mathrm{Li}_{2r}\left(e^{2\pi i\alpha}\right)$ and $\mathrm{Li}_{2r}\left(e^{2\pi i\beta}\right), r\in\mathbb{N}.$
\end{conjecture}
In the special case when $\alpha$ and $\beta$ are intergers, it is known from Zagier's work \cite[Theorem 9]{VZ} that $\zeta(k,\mathscr{B})+(-1)^k\zeta(k,\mathscr{B}^*)$ can be expressed as a linear combination of the values $\zeta(2r),\ r\in\mathbb{N}$.
\end{remark}

%Take an ideal $\mathscr{B}$ in the class $A^{-1}. $  There is a bijection map
%between the set of ideals of $A$ and the set of principal ideals $(\lambda)$ divisible by $\mathscr{B} $ given by 
%$$a\rightarrow a\mathscr{B}=\mbox{some principal ideal $(\lambda)$}.$$ 

%Take a {\em reduced} number $w\in K$, that it
%$$w>1, 0<w'<1.$$
%Then there are exactly 

%%%%%%%%%%%%%%%%%%%%%%%%%%%%%%%
\section{{Proofs}}
%%%%%%%%%%%%%%%%%%%%%%%%%%%%%%
%%%%%%%%%%%%%%%%%%%%%%%%%%%%%%
%\subsection{Asymptotic of $\psi_\beta(x)$ for large $x$}
\subsection{Proof of asymptotics}\label{asymp}

\begin{proof}[\textbf{Proof of Proposition \textup{\ref{special cases}}}]
The first part of the proposition follows easily from \cite[Proposition 2.4]{CK} and observing $\mathscr{F}(x;e^{2\pi i\alpha},e^{2\pi i\beta})=-\mathscr{F}_2(x;\alpha,\beta)+\frac{1}{x}\mathrm{Li}_2\left(e^{2\pi i\alpha}\right).$

We now prove the second part. Utilizing the representation of $\mathscr{F}_k(x;\alpha,\beta)$ in \eqref{hf} and the definition of the polylogarithm, we see that
\begin{align*}
&-\mathscr{F}_k(x;\alpha,\beta)-\left(\gamma+\log\left(1-e^{2\pi i\beta}\right)\right)\mathrm{Li}_{k-1}\left(e^{2\pi i\alpha}\right)\\
&=\sum_{p\geq1}\frac{e^{2\pi i\alpha p}}{p^{k-1}}\left\{\sum_{q\geq0}\left(-\frac{e^{2\pi i\beta q}}{px+q}\right)-\gamma-\log\left(1-e^{2\pi i\beta}\right)\right\}\\
&=\sum_{p\geq1}\frac{e^{2\pi i\alpha p}}{p^{k-1}}\left\{\sum_{q\geq0}\left(\frac{e^{2\pi i\beta (q+1)}}{1+q}-\frac{e^{2\pi i\beta q}}{px+q}\right)-\gamma\right\}.
\end{align*}
We obtain the result by taking the limit $\alpha,\beta\to0$ on both sides of the above equation, and using $\psi(x)=-\gamma+\sum_{q\geq0}\left(\frac{1}{1+q}-\frac{1}{px+q}\right)$ as well as \eqref{hk}.
\end{proof}

\begin{proof}[\textbf{Proof of Proposition \textup{\ref{ay}}}] 
%proof of \ref{asymptotic of psi_beta}
Note that
\begin{align}
\frac{1}{1-e^{-t}e^{2\pi i \beta}}=\sum_{n=0}^\infty \frac{a_n(\beta)}{n!}t^n,\nonumber
\end{align}
as $t\to0$. Then by Watson's lemma \cite[pp.~32--33, Theorem 2.4]{temme}, we have
\begin{align}
\int_0^\infty\frac{e^{-xt}}{1-e^{-t}e^{2\pi i \beta}}dt
\sim\sum_{n=0}^\infty\Gamma(n+1)\frac{a_n(\beta)}{n!}\frac{1}{x^{n+1}},\nonumber
\end{align}
as $x\to\infty$ in $|\arg(x)|\leq\pi-\delta$, $\delta>0$\footnote{Note that \cite[p.~33, Equation (2.7)]{temme} is valid for bigger region of $z$ through the argument given below of this equation. This is what we are using here.}. Using the definition of $\psi_\beta(x)$ from \eqref{poly}, we prove \eqref{asymptotic of psi_beta}.

\noindent
Invoking \eqref{asymptotic of psi_beta} in \eqref{higher herglotz}, we obtain, as $x\to\infty$,
\begin{align}
\mathscr{F}_{k}(x;\alpha,\beta)&\sim\sum_{p=1}^\infty\frac{e^{2\pi i p\alpha}}{p^{k-1}}\sum_{n=0}^\infty\frac{a_n(\beta)}{(px)^{n+1}}%\nonumber\\
 =\sum_{n=0}^\infty\frac{a_n(\beta)}{x^{n+1}}\sum_{p=1}^\infty\frac{e^{2\pi i p\alpha}}{p^{n+k}}.\nonumber
\end{align}
\\
\noindent Now we use the definition of polylogarithm function to arrive at \eqref{asymp of hnf large x}.

We next prove \eqref{asymp of hnf small x}. From \eqref{asymp of hnf large x}, as $x\to0$, it follows that
\begin{align}\label{cb}
(-x)^{k-2}\mathscr{F}_{k}\left(\frac{1}{x}; \alpha,\beta \right)\sim(-1)^{k-2}\sum_{n=0}^\infty a_n(\beta)\mathrm{Li}_{k+n}\left(e^{2\pi i\alpha}\right)x^{k+n-1}.
\end{align}
From the two-term functional equation \eqref{two term}, as $x\to0$, we have
\begin{align*}
 \mathscr{F}_{k}\left(x; \alpha,\beta \right)\sim&\frac{1}{x}\mathrm{Li}_{k}\left(e^{2\pi i\alpha}\right)-\frac{(-x)^{k-1}}{x+1}\mathrm{Li}_{k-r}\left(e^{2\pi i(\alpha+\beta)}\right)-(-x)^{k-2}\mathscr{F}_{k}\left(\frac{1}{x}; \beta, \alpha \right)\\
 &\quad+\sum_{r=1}^{k-1}(-x)^{r-1}\mathrm{Li}_{k-r}\left(e^{2\pi i\alpha}\right)\mathrm{Li}_{r}\left(e^{2\pi i\beta}\right).
\end{align*}
Equation \eqref{asymp of hnf small x} now follows upon employing \eqref{cb} in the above expression. 
\end{proof}

%%%%%%%%%%%%%%%%%%%%%%%%%%%%%%%%%%%%%%%%%%%%%%
%%%%%%%%%%%%%%%%%%%%%%55
%Two=Three term functional equation
%%%%%%%%%%%%%%
\medskip

%%%%%%%%%%%%%%%%%%%%%%%%%%%%%
\subsection{Proof of the two, three and six-term functional equations}
%%%%%%%%%%%%%%%%%%%%%%%%%%%%%

\begin{proof}[\textbf{Proof of Theorem \textup{\ref{fe thm}}}]
Let us introduce the notation
\begin{align}\label{F}
F_{m,n}(x,\alpha,\beta ):=(-x)^n\int_0^\infty\mathrm{Li}_m\left(e^{-xt}e^{2\pi i\alpha}\right)\mathrm{Li}_n\left(e^{-t}e^{2\pi i\beta}\right)dt,\quad(m,n\geq1).
\end{align}
Using the fact $\frac{d}{dt}\mathrm{Li}_{n+1}\left(e^{-t}e^{2\pi i\beta}\right)=-\mathrm{Li}_n\left(e^{-t}e^{2\pi i\beta}\right)$ in the above equation, we get
\begin{align}
F_{m+1,n}(x,\alpha,\beta )&=-(-x)^n\int_0^\infty\mathrm{Li}_{m+1}\left(e^{-xt}e^{2\pi i\alpha}\right)d\mathrm{Li}_{n+1}\left(e^{-t}e^{2\pi i\beta}\right)dt\nonumber\\
&=-(-x)^n\left(\sum_{p=1}^\infty\frac{e^{-pxt}e^{2\pi i\alpha p}}{p^{m+1}}\sum_{q=1}^\infty\frac{e^{-qt}e^{2\pi i\beta q}}{q^{n+1}}\right)\Bigg|_0^\infty\nonumber\\
&\quad+(-x)^{n+1}\int_0^\infty\mathrm{Li}_m\left(e^{-xt}e^{2\pi i\alpha}\right)\mathrm{Li}_{n+1}\left(e^{-t}e^{2\pi i\beta}\right)dt\nonumber\\
&=(-x)^n\mathrm{Li}_{m+1}\left(e^{2\pi i\alpha}\right)\mathrm{Li}_{n+1}\left(e^{2\pi i\beta}\right)\nonumber\\
&\quad+(-x)^{n+1}\int_0^\infty\mathrm{Li}_m\left(e^{-xt}e^{2\pi i\alpha}\right)\mathrm{Li}_{n+1}\left(e^{-t}e^{2\pi i\beta}\right)dt.\nonumber
\end{align}
This yields
\begin{align}\label{useful}
F_{m,n+1}(x,\alpha,\beta)&=F_{m+1,n}(x,\alpha,\beta)-(-x)^n\mathrm{Li}_{m+1}\left(e^{2\pi i\alpha}\right)\mathrm{Li}_{n+1}\left(e^{2\pi i\beta}\right).
\end{align}
Replacing $n$ by $n-1$ in the above equation, we obtain
\begin{align}
F_{m,n}(x,\alpha,\beta )&=F_{m+1,n-1}(x,\alpha,\beta )-(-x)^{n-1}\mathrm{Li}_{m+1}\left(e^{2\pi i\alpha}\right)\mathrm{Li}_{n}\left(e^{2\pi i\beta}\right).\nonumber
\end{align}
Recurring the above recurrence relation $n-1$ times, we arrive at
\begin{align}\label{F rec rel}
F_{m,n}(x,\alpha,\beta )&=F_{m+n-1,1}(x,\alpha,\beta)-\sum_{r=2}^n(-x)^{r-1}\mathrm{Li}_{m+n+1-r}\left(e^{2\pi i\alpha}\right)\mathrm{Li}_{r}\left(e^{2\pi i\beta}\right).
\end{align}
We make the change of variable $t\to t/x$ in \eqref{F} so that
\begin{align}\label{fe for capital F}
F_{m,n}(x,\alpha,\beta )+(-x)^{m+n-1}F_{m,n}\left(\frac{1}{x},\beta,\alpha  \right)=0.
\end{align}
We invoke \eqref{F rec rel} twice in \eqref{fe for capital F} and simplify to obtain
\begin{align}
&F_{m+n-1,1}(x,\alpha,\beta )+(-x)^{m+n-1}F_{m+n-1,1}\left(\frac{1}{x}. \beta,\alpha \right)\nonumber\\
&=\sum_{r=2}^n(-x)^{r-1}\mathrm{Li}_{m+n+1-r}\left(e^{2\pi i\alpha}\right)\mathrm{Li}_{r}\left(e^{2\pi i\beta}\right)+\sum_{r=2}^m(-x)^{m+n-r}\mathrm{Li}_{m+n+1-r}\left(e^{2\pi i\beta}\right)\mathrm{Li}_{r}\left(e^{2\pi i\alpha}\right)\nonumber\\
&=\sum_{r=2}^n(-x)^{r-1}\mathrm{Li}_{m+n+1-r}\left(e^{2\pi i\alpha}\right)\mathrm{Li}_{r}\left(e^{2\pi i\beta}\right)+\sum_{r=n+1}^{m+n-1}(-x)^{r-1}\mathrm{Li}_{r}\left(e^{2\pi i\beta}\right)\mathrm{Li}_{m+n+1-r}\left(e^{2\pi i\alpha}\right),\nonumber
\end{align}
where we replaced $r$ by $m+n+1-r$ in the second finite sum. Note that the finite sums on the right-hand side of the above equation can be combined together so that 
\begin{align}\label{xzl}
&F_{m+n-1,1}(x,\alpha,\beta )+(-x)^{m+n-1}F_{m+n-1,1}\left(\frac{1}{x}, \beta,\alpha \right)\nonumber\\
&=\sum_{r=2}^{m+n-1}(-x)^{r-1}\mathrm{Li}_{m+n+1-r}\left(e^{2\pi i\alpha}\right)\mathrm{Li}_{r}\left(e^{2\pi i\beta}\right).
\end{align}
Let $k=m+n+1$ and define, for $k\geq3$, 
\begin{align}\label{defn Fk}
F_k(x,\alpha,\beta ):=F_{k-2,1}(x,\alpha,\beta ).
\end{align}
Hence, from \eqref{xzl}, we get
\begin{align}\label{fe for F_k}
F_k(x,\alpha,\beta )+(-x)^{k-2}F_k\left(\frac{1}{x}, \beta,\alpha \right)&=\sum_{r=2}^{k-2}(-x)^{r-1}\mathrm{Li}_{k-r}\left(e^{2\pi i\alpha}\right)\mathrm{Li}_{r}\left(e^{2\pi i\beta}\right).
\end{align}
Now, we let $m=k-2$ and $n=1$ in \eqref{F} and use \eqref{defn Fk} to conclude that
\begin{align}\label{hh as a special case of double hh}
F_{k}(x,\alpha,\beta )&=-x\int_0^\infty\mathrm{Li}_{k-2}\left(e^{-xt}e^{2\pi i\alpha}\right)\mathrm{Li}_1\left(e^{-t}e^{2\pi i\beta}\right)dt\nonumber\\
&=x\sum_{p=1}^\infty\frac{e^{2\pi i\alpha}}{p^{k-2}}\int_0^\infty e^{-pxt}\log\left(1-e^{-t}e^{2\pi i\beta}\right)dt\nonumber\\
&=x\sum_{p=1}^\infty\frac{e^{2\pi i\alpha}}{p^{k-2}}\left\{\frac{e^{-pxt}}{-px}\log\left(1-e^{-t}e^{2\pi i\beta}\right)\Bigg|_{0}^\infty-\int_0^\infty \frac{e^{-pxt}}{-px}\frac{e^{-t}e^{2\pi i\beta}}{\left(1-e^{-t}e^{2\pi i\beta}\right)}dt\right\}\nonumber\\
&=\sum_{p=1}^\infty\frac{e^{2\pi i\alpha}}{p^{k-1}}\left\{\log\left(1-e^{2\pi i\beta}\right)+\int_0^\infty\frac{ e^{-pxt}}{e^{t}e^{-2\pi i\beta}-1}dt\right\}\nonumber\\
&=\log\left(1-e^{2\pi  i\beta}\right)\mathrm{Li}_{k-1}\left(e^{2\pi i\alpha}\right)+\sum_{p=1}^\infty\frac{e^{2\pi i\alpha}}{p^{k-1}}\left\{\int_0^\infty\frac{ e^{-pxt}dt}{1-e^{-t}e^{2\pi i\beta}}-\int_0^\infty e^{-pxt}dt\right\}\nonumber\\
&=\log\left(1-e^{2\pi  i\beta}\right)\mathrm{Li}_{k-1}\left(e^{2\pi i\alpha}\right)+\mathscr{F}_{k}(x; \alpha,\beta )-\frac{1}{x}\mathrm{Li}_k\left(e^{2\pi i\alpha}\right),
\end{align}
where we used the definition of $\mathscr{F}_{k}(x;\alpha,\beta )$ from \eqref{int repres for hh}.
Substitute the values from \eqref{hh as a special case of double hh} in \eqref{fe for F_k} to obtain
\begin{align}
&\mathscr{F}_{k}(x;\alpha,\beta )+(-x)^{k-2}\mathscr{F}_{k}\left(\frac{1}{x};\beta,\alpha\right) \nonumber\\
&=\frac{1}{x}\mathrm{Li}_k\left(e^{2\pi i\alpha}\right)+(-x)^{k-2}x\mathrm{Li}_k\left(e^{2\pi i\beta}\right)-\log\left(1-e^{2\pi  i\beta}\right)\mathrm{Li}_{k-1}\left(e^{2\pi i\alpha}\right)\nonumber\\
&\quad-(-x)^{k-2}\log\left(1-e^{2\pi  i\alpha}\right)\mathrm{Li}_{k-1}\left(e^{2\pi i\beta}\right)+\sum_{r=2}^{k-2}(-x)^{r-1}\mathrm{Li}_{k-r}\left(e^{2\pi i\alpha}\right)\mathrm{Li}_{r}\left(e^{2\pi i\beta}\right).\nonumber
\end{align}
The two-term functional equation \eqref{two term} now follows by observing the fact that the third and fourth terms of the right-hand side are $r=1$ and $r=k-1$ terms of the finite sum, respectively.

%\subsection{Three-term functional equation for $\mathscr{F}_{\alpha,\beta,k}(x)$.}
 \medskip

We next prove the three-term \eqref{three term} and six-term functional equation \eqref{six term}.

%\ref{hh three term fe} 
Let 
\begin{align}\label{G}
G_{m,n}(x,\alpha,\beta ):=(-x)^n\int_0^\infty\mathrm{Li}_{m,n}\left(e^{-xt}e^{2\pi i\alpha},e^{-t}e^{2\pi i\beta}\right)dt,
\end{align}
where $\mathrm{Li}_{m,n}(x,y)$ is defined in \eqref{double polylog}.
%\begin{align}\label{double polylog}
%\mathrm{Li}_{m,n}(x,y):=\sum_{0<p<q}\frac{x^py^q}{p^mq^n}.
%\end{align}
It is easy to find that
\begin{align*}
\frac{d}{dt}\mathrm{Li}_{m+1,n+1}\left(e^{-xt}e^{2\pi i\alpha},e^{-t}e^{2\pi i\beta}\right)&=-x\mathrm{Li}_{m,n+1}\left(e^{-xt}e^{2\pi i\alpha},e^{-t}e^{2\pi i\beta}\right)\\
&\quad-\mathrm{Li}_{m+1,n}\left(e^{-xt}e^{2\pi i\alpha},e^{-t}e^{2\pi i\beta}\right).
\end{align*}
This implies the following relation
\begin{align}\label{derivative}
\mathrm{Li}_{m+1,n}\left(e^{-xt}e^{2\pi i\alpha},e^{-t}e^{2\pi i\beta}\right)&=-x\mathrm{Li}_{m,n+1}\left(e^{-xt}e^{2\pi i\alpha},e^{-t}e^{2\pi i\beta}\right)\nonumber\\
&\quad-\frac{d}{dt}\mathrm{Li}_{m+1,n+1}\left(e^{-xt}e^{2\pi i\alpha},e^{-t}e^{2\pi i\beta}\right).
\end{align}
It follows from \eqref{G} and \eqref{derivative} that
\begin{align}\label{original reccur relation}
G_{m+1,n}(x,\alpha,\beta )&=(-x)^{n+1}\int_0^\infty\mathrm{Li}_{m,n+1}\left(e^{-xt}e^{2\pi i\alpha},e^{-t}e^{2\pi i\beta}\right)dt\nonumber\\
&\qquad-(-x)^n\int_0^\infty\frac{d}{dt}\mathrm{Li}_{m+1,n+1}\left(e^{-xt}e^{2\pi i\alpha},e^{-t}e^{2\pi i\beta}\right)dt\nonumber\\
%&=G_{m,n+1}(x,\alpha,\beta )+(-x)^n\sum_{0<p<q}\frac{e^{2\pi i\alpha p}e^{2\pi i\beta q}}{p^{m+1}q^{n+1}}\nonumber\\
&=G_{m,n+1}(x,\alpha,\beta )+(-x)^n\mathrm{Li}_{m+1,n+1}(e^{2\pi i \alpha},e^{2\pi i \beta}).
\end{align}
The above equation gives
\begin{align}
G_{m,n+1}(x, \alpha,\beta )=G_{m+1,n}(x, \alpha,\beta )-(-x)^n 
\mathrm{Li}_{m+1,n+1}(e^{2\pi i \alpha},e^{2\pi i \beta}),\nonumber
\end{align}
which, upon replacing $n$ by $n-1$, immediately yields
\begin{align}
G_{m,n}(x,\alpha,\beta )=G_{m+1,n-1}(x, \alpha,\beta )-(-x)^{n-1}\mathrm{Li}_{m+1,n}(e^{2\pi i \alpha},e^{2\pi i \beta}).\nonumber
\end{align}
Recurring the above relation $n-1$ times, we obtain
\begin{align}\label{rec for G}
G_{m,n}(x, \alpha,\beta )=G_{m+n-1,1}(x, \alpha,\beta )-\sum_{r=2}^{n}(-x)^{r-1}\mathrm{Li}_{m+n+1-r,r}(e^{2\pi i\alpha}, e^{2\pi i \beta} ).
\end{align}
Using \eqref{F} and the definition of the polylogarithm function, we get
\begin{align}\label{f g g}
F_{m,n}(x,\alpha,\beta )&=(-x)^n\int_0^\infty\sum_{p,q=1}^\infty\frac{e^{-pxt}e^{2\pi ip\alpha}e^{-qt}e^{2\pi iq\beta}}{p^mq^n}\nonumber\\
&=(-x)^n\int_0^\infty\left(\sum_{0<p<q}+\sum_{p>q>0}\right)\frac{e^{-pxt}e^{2\pi ip\alpha}e^{-qt}e^{2\pi iq\beta}}{p^mq^n}\nonumber\\
&\quad+(-x)^n\int_0^\infty\sum_{p=1}^\infty\frac{e^{-pxt-pt}e^{2\pi i(\alpha+\beta)p}}{p^{m+n}}\nonumber\\
&=(-x)^n\int_0^\infty \mathrm{Li}_{m,n}\left(e^{-xt}e^{2\pi i\alpha},e^{-t}e^{2\pi i\beta}\right)dt\nonumber\\
&\quad+(-x)^n\int_0^\infty \mathrm{Li}_{n,m}\left(e^{-t}e^{2\pi i\beta},e^{-xt}e^{2\pi i\alpha}\right)dt+\frac{(-x)^n}{x+1}\sum_{p=1}^\infty\frac{e^{2\pi i(\alpha+\beta)p}}{p^{m+n+1}}\nonumber\\
&=G_{m,n}(x,\alpha,\beta )+\frac{(-x)^n}{x}\int_0^\infty\mathrm{Li}_{n,m}\left(e^{-t/x}e^{2\pi i\beta},e^{-t}e^{2\pi i\alpha}\right)dt\nonumber\\
&\quad+\frac{(-x)^n}{x+1}\mathrm{Li}_{m+n+1}\left(e^{2\pi i(\alpha+\beta)}\right)\nonumber\\
&=G_{m,n}(x,\alpha,\beta )-(-x)^{m+n-1}G_{n,m}\left(\frac{1}{x},\beta,\alpha \right)+\frac{(-x)^n}{x+1}\mathrm{Li}_{m+n+1}\left(e^{2\pi i(\alpha+\beta)}\right).
\end{align}
Hence invoking \eqref{rec for G} in \eqref{f g g}, we obtain
\begin{align}\label{before multiple}
F_{m,n}(x,\alpha,\beta )&=G_{m+n-1,1}(x,\alpha,\beta )-(-x)^{m+n-1}G_{m+n-1,1}\left(\frac{1}{x}, \beta,\alpha \right)\nonumber\\
&\quad+\frac{(-x)^n}{x+1}\mathrm{Li}_{m+n+1}\left(e^{2\pi i(\alpha+\beta)}\right)-\sum_{r=2}^{n}(-x)^{r-1}\mathrm{Li}_{m+n+1-r,r}( e^{2\pi i \alpha},e^{2\pi i\beta})\nonumber\\
&\quad+(-x)^{m+n-1}\sum_{r=2}^{m}(-x)^{1-r}\mathrm{Li}_{m+n+1-r,r}( e^{2\pi i \beta},e^{2\pi i\alpha} ).
\end{align}
From \eqref{double polylog}, it is straightforward to see 
\begin{align}\label{multiple zeta value}
\mathrm{Li}_{m}\left(e^{2\pi i\alpha}\right)\mathrm{Li}_{k-m}\left(e^{2\pi i\beta}\right)&=\sum_{p,q=1}^\infty\frac{e^{2\pi i\alpha p}e^{2\pi i\beta q}}{p^mq^{k-m}}\nonumber\\
&=\left\{\sum_{0<p<q}+\sum_{p>q>0}+\sum_{p=q}\right\}\frac{e^{2\pi i(\alpha p+\beta q)}}{p^mq^{k-m}}\nonumber\\
&=\mathrm{Li}_{m,k-m}(e^{2\pi i \alpha},e^{2 \pi i \beta} )+\mathrm{Li}_{k-m,m}(e^{2\pi i \beta}, e^{2 \pi i \alpha})+\mathrm{Li}_k\left(e^{2\pi i(\alpha+\beta)}\right).
\end{align}
%\begin{align}\label{multiple zeta value}
%\mathrm{Li}_{k-m}\left(e^{2\pi i\alpha}\right)\mathrm{Li}_{m}\left(e^{2\pi i\beta}\right)&=\sum_{p,q=1}^\infty\frac{e^{2\pi i\alpha p}e^{2\pi i\beta q}}{p^{k-m}q^{m}}\nonumber\\
%&=\left\{\sum_{0<p<q}+\sum_{p>q>0}+\sum_{p=q}\right\}\frac{e^{2\pi i(\alpha p+\beta q)}}{p^{k-m}q^{m}}\nonumber\\
%&=\zeta_{\alpha,\beta}(k-m,m)+\zeta_{\beta,\alpha}(m,k-m)+\mathrm{Li}_k\left(e^{2\pi i(\alpha+\beta)}\right).
%\end{align}
Equations \eqref{before multiple} and \eqref{multiple zeta value} together yield
\begin{align}
&F_{m,n}(x,\alpha,\beta )\nonumber\\
&=G_{m+n-1,1}(x,\alpha,\beta )-(-x)^{m+n-1}G_{m+n-1,1}\left(\frac{1}{x}. \beta,\alpha \right)+\frac{(-x)^n}{x+1}\mathrm{Li}_{m+n+1}\left(e^{2\pi i(\alpha+\beta)}\right)\nonumber\\
&\quad-\sum_{r=2}^{n}(-x)^{r-1}\left\{-\mathrm{Li}_{r,m+n+1-r}( e^{2\pi i \beta},e^{2 \pi i \alpha})-\mathrm{Li}_{m+n+1}\left(e^{2\pi i(\alpha+\beta)}\right)\right.\nonumber\\
&\left.\quad +\mathrm{Li}_r\left(e^{2\pi i\beta}\right)\mathrm{Li}_{m+n+1-r}\left(e^{2\pi i\alpha}\right)\right\}+\sum_{r=2}^{m}(-x)^{m+n-r}\mathrm{Li}_{m+n+1-r,r}
(e^{2\pi i\beta}, e^{2\pi i \alpha} )\nonumber\\
&=G_{m+n-1,1}(x,\alpha,\beta )-(-x)^{m+n-1}G_{m+n-1,1}\left(\frac{1}{x},\beta,\alpha \right)+\frac{(-x)^n}{x+1}\mathrm{Li}_{m+n+1}\left(e^{2\pi i(\alpha+\beta)}\right)\nonumber\\
&\quad+\sum_{r=2}^{n}(-x)^{r-1}\mathrm{Li}_{r,m+n+1-r}( e^{2\pi i\beta}, e^{2\pi i \alpha})+\sum_{r=n+1}^{m+n-1}(-x)^{r-1}\mathrm{Li}_{r,m+n+1-r}(e^{2\pi i\beta}, e^{2\pi i \alpha})\nonumber\\
&\quad+\mathrm{Li}_{m+n+1}\left(e^{2\pi i(\alpha+\beta)}\right)\sum_{r=2}^n(-x)^{r-1}-\sum_{r=2}^n(-x)^{r-1}\mathrm{Li}_r\left(e^{2\pi i\beta}\right)\mathrm{Li}_{m+n+1-r}\left(e^{2\pi i\alpha}\right).\nonumber
\end{align}
Rearranging the terms and simplifying further in the above equation, we obtain
\begin{align}\label{F in G G}
&F_{m,n}(x,\alpha,\beta )+\sum_{r=2}^n(-x)^{r-1}\mathrm{Li}_r\left(e^{2\pi i\beta}\right)\mathrm{Li}_{m+n+1-r}\left(e^{2\pi i\alpha}\right)\nonumber\\
&=G_{m+n-1,1}(x,\alpha,\beta )-(-x)^{m+n-1}G_{m+n-1,1}\left(\frac{1}{x},\beta,\alpha \right)-\frac{x+(-x)^n}{x+1}\mathrm{Li}_{m+n+1}\left(e^{2\pi i(\alpha+\beta)}\right)\nonumber\\
&\quad+\sum_{r=2}^{m+n-1}(-x)^{r-1}\mathrm{Li}_{r,m+n+1-r}(e^{2\pi i \beta},e^{2\pi i \alpha})+\frac{(-x)^n}{x+1}\mathrm{Li}_{m+n+1}\left(e^{2\pi i(\alpha+\beta)}\right)\nonumber\\
&=G_{m+n-1,1}(x,\alpha,\beta )-(-x)^{m+n-1}G_{m+n-1,1}\left(\frac{1}{x}, \beta,\alpha \right)-\frac{x}{x+1}\mathrm{Li}_{m+n+1}\left(e^{2\pi i(\alpha+\beta)}\right)\nonumber\\
&\quad+\sum_{r=2}^{m+n-1}(-x)^{r-1}\mathrm{Li}_{r,m+n+1-r}(e^{2\pi i \beta},e^{2\pi i \alpha}).
\end{align}
Invoke \eqref{F rec rel} in \eqref{F in G G} to see that
\begin{align}
F_{m+n-1,1}(x,\alpha,\beta )&=G_{m+n-1,1}(x,\alpha,\beta)-(-x)^{m+n-1}G_{m+n-1,1}\left(\frac{1}{x}, \beta,\alpha \right)\nonumber\\
&+\sum_{r=2}^{m+n-1}(-x)^{r-1} \mathrm{Li}_{r,m+n+1-r}(e^{2\pi i \beta},e^{2\pi i \alpha})-\frac{x}{x+1}\mathrm{Li}_{m+n+1}\left(e^{2\pi i(\alpha+\beta)}\right).\nonumber
\end{align}
Letting $k=m+n+1$ and using \eqref{defn Fk}, we deduce that
\begin{align}\label{observe}
F_{k}(x,\alpha,\beta, )&=G_{k-2,1}(x,\alpha,\beta )-(-x)^{k-2}G_{k-2,1}\left(\frac{1}{x},\beta,\alpha \right)\nonumber\\
&\quad+\sum_{r=2}^{k-2}(-x)^{r-1}\mathrm{Li}_{r,k-r}(e^{2\pi i \beta},e^{2\pi i \alpha})-\frac{x}{x+1}\mathrm{Li}_{k}\left(e^{2\pi i(\alpha+\beta)}\right).
\end{align}
Now, observe that
\begin{align}
\mathrm{Li}_{k-1,0}\left(e^{-xt}e^{2\pi i\alpha},e^{-t}e^{2\pi i\beta}\right)&=\sum_{0<p<q}\frac{e^{-pxt}e^{2\pi i\alpha p}e^{-qt}e^{2\pi i\beta q}}{p^{k-1}}\nonumber\\
&=\sum_{m=1}^\infty\sum_{p=1}^\infty \frac{e^{-pxt}e^{2\pi i\alpha p}e^{-(m+p)t}e^{2\pi i\beta (m+p)}}{p^{k-1}}\nonumber\\
&=\sum_{m=1}^\infty e^{-mt}e^{2\pi i\beta m}\sum_{p=1}^\infty\frac{e^{-p(x+1)t}e^{2\pi i(\alpha+\beta)p}}{p^{k-1}}\nonumber\\
&=\mathrm{Li}_0\left(e^{-t}e^{2\pi i\beta}\right)\mathrm{Li}_{k-1}\left(e^{-(x+1)t}e^{2\pi i(\alpha+\beta)}\right).\nonumber
\end{align}
Integrating the above equation, we obtain
\begin{align}\label{int of G}
G_{k-1,0}(x,\alpha,\beta )&=\int_0^\infty \mathrm{Li}_{k-1}\left(e^{-(x+1)t}e^{2\pi i(\alpha+\beta)}\right)\mathrm{Li}_0\left(e^{-t}e^{2\pi i\beta}\right)dt.
\end{align}
Employing \eqref{useful} with letting $m=k-2, n=0$ and replacing $x$ by $x+1$ and $\alpha$  by $\alpha+\beta$, we have
\begin{align}\label{int eval in terms of F}
&\int_0^\infty \mathrm{Li}_{k-1}\left(e^{-(x+1)t}e^{2\pi i(\alpha+\beta)}\right)\mathrm{Li}_0\left(e^{-t}e^{2\pi i\beta}\right)dt\nonumber\\
&=F_{k-2,1}(x+1,\alpha+\beta,\beta )+\mathrm{Li}_{k-1}\left(e^{2\pi i(\alpha+\beta)}\right)\mathrm{Li}_1\left(e^{2\pi i\beta}\right).
\end{align}
Substituting \eqref{int eval in terms of F} in \eqref{int of G} and using \eqref{defn Fk}, we are led to
\begin{align}\label{gk-1}
G_{k-1,0}(x, \alpha,\beta )=F_{k}(x+1, \alpha+\beta,\beta )+\mathrm{Li}_{k-1}\left(e^{2\pi i(\alpha+\beta)}\right)\mathrm{Li}_1\left(e^{2\pi i\beta}\right).
\end{align}
Evaluate $G_{k-2,1}(x,\alpha,\beta )$ from \eqref{original reccur relation} with letting $m=k-2$ and $n=0$ and then in the resulting expression substitute value from \eqref{gk-1} to arrive at
\begin{align}\label{hh as a special case of G}
G_{k-2,1}(x,\alpha,\beta )=F_{k}(x+1,\alpha+\beta,\beta  )+\mathrm{Li}_{k-1}\left(e^{2\pi i(\alpha+\beta)}\right)\mathrm{Li}_1\left(e^{2\pi i\beta}\right)-\mathrm{Li}_{k-1,1}(e^{2 \pi i \alpha}, e^{2 \pi i \beta}).
\end{align}
Equations \eqref{observe} and \eqref{hh as a special case of G} together yield
\begin{align}\label{zetaalpha beta}
&F_{k}(x,\alpha,\beta )-F_{k}(x+1,\alpha+\beta,\beta )+(-x)^{k-2}F_{k}\left(\frac{x+1}{x},\alpha+\beta,\alpha \right)\nonumber\\
&=\mathrm{Li}_{k-1}\left(e^{2\pi i(\alpha+\beta)}\right)\mathrm{Li}_{1}\left(e^{2\pi i\beta}\right)-\mathrm{Li}_{k-1,1} (e^{2\pi i\alpha}, e^{2\pi i \beta} )-(-x)^{k-2}\mathrm{Li}_{k-1}\left(e^{2\pi i(\alpha+\beta)}\right)\mathrm{Li}_{1}\left(e^{2\pi i\alpha}\right)\nonumber\\
&\quad+(-x)^{k-2}\mathrm{Li}_{k-1,1}(e^{2 \pi i \beta}, e^{2 \pi i \alpha} )+\sum_{r=2}^{k-2}(-x)^{r-1}\mathrm{Li}_{r,k-r} (e^{2 \pi i \beta}, e^{2 \pi i \alpha})-\frac{x}{x+1}\mathrm{Li}_{k}\left(e^{2\pi i(\alpha+\beta)}\right).
\end{align}
From \eqref{multiple zeta value}, we have
\begin{align}
\mathrm{Li}_{k-1,1}( e^{2 \pi i \alpha}, e^{2 \pi i \beta})
=-\mathrm{Li}_{1,k-1} ( e^{2 \pi i \beta}, e^{2 \pi i \alpha})-\mathrm{Li}_{k}\left(e^{2\pi i(\alpha+\beta)}\right)+\mathrm{Li}_{k-1}\left(e^{2\pi i\alpha}\right)\mathrm{Li}_{1}\left(e^{2\pi i\beta}\right).\nonumber
\end{align}
Using the above relation in \eqref{zetaalpha beta} and  noting that $\mathrm{Li}_{1,k-1}(e^{2 \pi i \beta}, e^{2 \pi i \alpha})$ and \newline $(-x)^{k-2}\mathrm{Li}_{k-1,1}(e^{2 \pi i \beta}, e^{2 \pi i \alpha})$ are $r=1$ and $r=k-1$ terms of the finite sum respectively, we deduce that
\begin{align}
&F_{k}(x,\alpha,\beta )-F_{k}(x+1,\alpha+\beta,\beta )+(-x)^{k-2}F_{k}\left( \frac{x+1}{x}, \alpha+\beta,\alpha\right)\nonumber\\
&=\mathrm{Li}_{k-1}\left(e^{2\pi i(\alpha+\beta)}\right)\mathrm{Li}_{1}\left(e^{2\pi i\beta}\right)-(-x)^{k-2}\mathrm{Li}_{k-1}\left(e^{2\pi i(\alpha+\beta)}\right)\mathrm{Li}_{1}\left(e^{2\pi i\alpha}\right)\nonumber\\
&\quad+\frac{1}{x+1}\mathrm{Li}_{k}\left(e^{2\pi i(\alpha+\beta)}\right)-\mathrm{Li}_{k-1}\left(e^{2\pi i\alpha}\right)\mathrm{Li}_{1}\left(e^{2\pi i\beta}\right)+\sum_{r=1}^{k-1}(-x)^{r-1}\mathrm{Li}_{r,k-r}(e^{2 \pi i \beta}, e^{2 \pi i \alpha}).\nonumber
\end{align}
The three-term functional equation \eqref{three term} now follows upon invoking \eqref{hh as a special case of double hh} in the above equation and simplifying the terms.
 
Finally six-term relation \eqref{six term} immediately comes from three-term relation \eqref{three term} and using the relation
$$F_{k}|_{2k}(-I)(x,\alpha,\beta ) = F_{k}(x,-\alpha, -\beta ).$$
This completes the proof of Theorem \ref{fe thm}.

\end{proof}
 
\medskip

%%%%%%%%%%%%%%%%%%%%%%%%%
\subsection{Proof of the second Kronecker limit formula}
%%%%%%%%%%%%%%%%%%%%%%
  %proof of main theorem
%%%%%%%%%%%%%%%%%%%%%%%%%%%
\begin{proof}[\textbf{Proof of Theorem \textup{\ref{mz}}}]

We first prove the first part of the theorem. More precisely, we show that the function $Z_Q(s;(\alpha,\beta))$ is an analytic function of $s$ in the region Re$(s)>\frac{1}{2}$.

To proceed, we first rewrite $Z_Q(s;(\alpha,\beta))$ in an equivalent form. Let $Q(x,y)=ax^2+bxy+cy^2$ denote the quadratic form from \eqref{epstein zeta}. Solving the quadratic equation $Q(1,-x)=0$, that is, $cx^2-bx+a=0$ yields the roots $w=\left(b+\sqrt{b^2-4ac}\right)/(2c)$ and $w'=\left(b-\sqrt{b^2-4ac}\right) /(2c)$. In particular, we have the identities $ww'=a/c,\ w+w'=b/c$ and $w-w'=1/c$. Hence, using these facts, one can see that
\begin{align*}
\frac{1}{w-w'}(y+xw)(y+xw')&=\frac{y^2+(w+w')xy+x^2ww'}{w-w'}\\
&=c\left(y^2+\frac{b}{c}xy+\frac{a}{c}x^2\right)\\
&=ax^2+bxy+cy^2=Q(x,y).
\end{align*}
Therefore, the zeta function $Z_Q(s;(\alpha,\beta))$ becomes
\begin{align}\label{epstein alternate}
Z_Q(s;(\alpha,\beta))=\sum_{p\geq1}\sum_{q\geq0}e^{2\pi i(p\alpha+q\beta)}\frac{(w-w')^s}{(pw+q)^s(pw'+q)^s}.
\end{align}

Since the series in \eqref{zagier zeta function} converges absolutely for 
Re$(s)>1$, the same holds for the above series in this region. We now prove the analyticity of the series in \eqref{epstein zeta} (or \eqref{epstein alternate}) in the extended half-plane Re$(s)>1/2$. To that end, let us consider the partial sums
\begin{align}\label{partial sum}
\sum_{1\leq p\leq m_1}\sum_{0\leq q\leq m_2}\frac{e^{2\pi ip\alpha}e^{2\pi iq\beta}}{(pw+q)^s(pw'+q)^s}.
\end{align}
 \\
Since 
\begin{eqnarray}\label{difference}
\displaystyle c_{q+1}-c_q=e^{2\pi i\beta q} 
\, \, \mbox{\, and 
$\displaystyle |c_q|<\frac{1}{|e^{2\pi i\beta}-1|}$ \,\,
with
$\displaystyle c_q:=\frac{e^{2\pi i\beta q}}{e^{2\pi i\beta }-1},$}
\end{eqnarray}
we can write
\begin{align}\label{equality}
&\sum_{1\leq p\leq m_1}\sum_{0\leq q\leq m_2}\frac{e^{2\pi ip\alpha}e^{2\pi iq\beta}}{(pw+q)^s(pw'+q)^s}\nonumber\\
&=\sum_{1\leq p\leq m_1}e^{2\pi ip\alpha}\left\{\sum_{1\leq q\leq m_2}c_q\left(\frac{1}{(pw+q-1)^s(pw'+q-1)^s}-\frac{1}{(pw+q)^s(pw'+q)^s}\right)\right.\nonumber\\
&\qquad\left.+\frac{c_{m_2+1}}{(pw+m_2)^s(pw'+m_2)^s}-\frac{c_0}{p^{2s}(ww')^s}\right\}.
\end{align} 
It is easy to see that since $p\geq1,\ m_2\geq1,\ w>0,\ w'>0$, and Re$(s)>0$, we have
\begin{align}
\left|(pw+m_2)^s(pw'+m_2)^s\right|&=\left|(p^2ww'+pm_2(w+w')+m_2^2)^s\right|\nonumber\\
&\geq\left|(p^2ww'+pm_2(w+w'))^s\right|\nonumber\\
&\geq\left|(p^2ww')^s\right|\nonumber\\
&=(ww')^{\mathrm{Re}(s)}p^{2\mathrm{Re}(s)}.\nonumber
\end{align}
Combining this with equation \eqref{difference}, we arrive at 
\begin{align}\label{inequality2}
\left|\frac{c_{m_2+1}}{(pw+m_2)^s(pw'+m_2)^s}\right|<\frac{1}{\left|e^{2\pi i\beta}-1\right|}\frac{1}{(ww')^{\mathrm{Re}(s)}}\frac{1}{p^{2\mathrm{Re}(s)}}.
\end{align}
We know that for any continuous function $\phi(u)$,
\begin{align*}
\phi(m-1)-\phi(m)=\int_{m-1}^m\phi'(u)du.
\end{align*}
Applying this identity with $\displaystyle\phi(u)=\frac{1}{(pw+u)^s(pw'+u)^s}$ and $m=q$, we obtain
\begin{align}\label{inequality3}
&\left|\frac{1}{(pw+q-1)^s(pw'+q-1)^s}-\frac{1}{(pw+q)^s(pw'+q)^s}\right|\nonumber\\
&=\left|\int_{q-1}^q\frac{-s((p(w+w')+2u))}{((pw+u)(pw'+u))^{s+1}}\right|du\nonumber\\
&<2|s|\int_{q-1}^q\frac{p(w+w')}{\left(p^2ww'+pu(w+w')+u^2\right)^{\mathrm{Re}(s)+1}}du+\int_{q-1}^q\frac{2|s| udu}{\left(p^2ww'+pu(w+w')+u^2\right)^{\mathrm{Re}(s)+1}}.
\end{align}
Hence, from \eqref{equality}, \eqref{inequality2} and \eqref{inequality3}, we see that
\begin{align}\label{partial sum finite}
&\left|\sum_{1\leq p\leq m_1}\sum_{0\leq q\leq m_2}\frac{e^{2\pi ip\alpha}e^{2\pi iq\beta}}{(pw+q)^s(pw'+q)^s}\right|\nonumber\\
&< \sum_{1\leq p\leq m_1}\left|e^{2\pi ip\alpha}\right|\left\{\sum_{1\leq q\leq m_2}\frac{2|s|}{|e^{2\pi i\beta}-1|}\left[\int_{q-1}^q\frac{p(w+w')}{\left(p^2ww'+pu(w+w')+u^2\right)^{\mathrm{Re}(s)+1}}du\right.\right.\nonumber\\
&\quad\left.\left.+\int_{q-1}^q\frac{u}{\left(p^2ww'+pu(w+w')+u^2\right)^{\mathrm{Re}(s)+1}}du\right]+\frac{1}{|e^{2\pi i\beta}-1|}\frac{1}{(ww'p^2)^{\mathrm{Re}(s)}}\right.\nonumber\\
&\quad\left.+\frac{1}{|e^{2\pi i\beta}-1|}\frac{1}{(ww'p^2)^{\mathrm{Re}(s)}}\right\}\nonumber\\
&=\frac{2|s|}{|e^{2\pi i\beta}-1|}\left\{\sum_{1\leq p\leq m_1}(w+w')\int_{0}^{m_2}\frac{p}{\left(p^2ww'+pu(w+w')+u^2\right)^{\mathrm{Re}(s)+1}}du\right.\nonumber\\
&\left.\quad+\sum_{1\leq p\leq m_1}\int_{0}^{m_2}\frac{u}{\left(p^2ww'+pu(w+w')+u^2\right)^{\mathrm{Re}(s)+1}}du\right\}\nonumber\\
&\quad+\frac{2}{|e^{2\pi i\beta}-1|}\frac{1}{(ww')^{\mathrm{Re}(s)}}\sum_{1\leq p\leq m_1}\frac{1}{p^{2\mathrm{Re}(s)}}.
\end{align}
We next show that, when $m_1\to\infty$ and $m_2\to\infty$,  these partial sums, which are entire functions of $s$, converge uniformly in every compact subset of the half-plane Re$(s)>\frac{1}{2}$. To that end, we majorize these partial sums by functions that depend only on the values of $s$  in compact subsets of the region Re$(s)>\frac{1}{2}$. 

Observe that the right-hand side of \eqref{partial sum finite} can be estimated by
\begin{align}
&\frac{2|s|}{|e^{2\pi i\beta}-1|}\left\{(w+w')\sum_{1\leq p\leq m_1}\int_{0}^{\infty}\frac{p}{\left(p^2ww'+pu(w+w')+u^2\right)^{\mathrm{Re}(s)+1}}du\right.\nonumber\\
&\left.+\sum_{1\leq p\leq m_1}\int_{0}^{\infty}\frac{u}{\left(p^2ww'+pu(w+w')+u^2\right)^{\mathrm{Re}(s)+1}}du\right\}+\frac{2}{|e^{2\pi i\beta}-1|}\frac{1}{(ww')^{\mathrm{Re}(s)}}\sum_{1\leq p\leq m_1}\frac{1}{p^{2\mathrm{Re}(s)}}\nonumber\\
&= \frac{2|s|}{|e^{2\pi i\beta}-1|}\left\{(w+w')\sum_{1\leq p\leq m_1}\frac{1}{p^{2\mathrm{Re}(s)}}\int_{0}^{\infty}\frac{1}{\left(ww'+u(w+w')+u^2\right)^{\mathrm{Re}(s)+1}}du\right.\nonumber\\
&\left.+\sum_{1\leq p\leq m_1}\frac{1}{p^{2\mathrm{Re}(s)}}\int_{0}^{\infty}\frac{u}{\left(ww'+u(w+w')+u^2\right)^{\mathrm{Re}(s)+1}}du\right\}+\frac{2}{|e^{2\pi i\beta}-1|}\frac{1}{(ww')^{\mathrm{Re}(s)}}\sum_{1\leq p\leq m_1}\frac{1}{p^{2\mathrm{Re}(s)}},\nonumber
\end{align}
where we employed the change of variable $u\to pu$ in both integrals. We now use the simple inequality $ww'+u(w+w')+u^2>u^2+\epsilon,\ \epsilon>0$ which holds true for $w>w'>0$ and $u\geq0$. Thus, for Re$(s)>1/2$,
\begin{align}
&\frac{2|s|}{|e^{2\pi i\beta}-1|}\left\{(w+w')\sum_{1\leq p\leq m_1}\int_{0}^{\infty}\frac{p}{\left(p^2ww'+pu(w+w')+u^2\right)^{\mathrm{Re}(s)+1}}du\right.\nonumber\\
&\left.+\sum_{1\leq p\leq m_1}\int_{0}^{\infty}\frac{u}{\left(p^2ww'+pu(w+w')+u^2\right)^{\mathrm{Re}(s)+1}}du\right\}+\frac{2}{|e^{2\pi i\beta}-1|}\frac{1}{(ww')^{\mathrm{Re}(s)}}\sum_{1\leq p\leq m_1}\frac{1}{p^{2\mathrm{Re}(s)}}\nonumber\\
&< \frac{2|s|}{|e^{2\pi i\beta}-1|}\left\{(w+w')\int_{0}^{\infty}\frac{1}{(u^2+\epsilon)^{\mathrm{Re}(s)+1}}du\sum_{1\leq p\leq m_1}\frac{1}{p^{2\mathrm{Re}(s)}}\right.\nonumber\\
&\left.\quad+\int_{0}^{\infty}\frac{u}{(u^2+\epsilon)^{\mathrm{Re}(s)+1}}du\sum_{1\leq p\leq m_1}\frac{1}{p^{2\mathrm{Re}(s)}}\right\}+\frac{2}{|e^{2\pi i\beta}-1|}\frac{1}{(ww')^{\mathrm{Re}(s)}}\sum_{1\leq p\leq m_1}\frac{1}{p^{2\mathrm{Re}(s)}}\nonumber\\
&\ll A\sum_{1\leq p\leq m_1}\frac{1}{p^{2\mathrm{Re}(s)}},\nonumber
\end{align}
which follows from the fact that the both integrals on the right-hand side in the penultimate step are also convergent when Re$(s)>1/2$, and where the constant $A$ depends only on the values of $s$ lying in compact subsets of the region Re$(s)>1/2$ and on the fixed parameters $\alpha,\beta,w$ and $w'$. Hence, when we sum over  $p$ from 1 to $\infty$ and $q$ from $0$ to $\infty$ in \eqref{partial sum finite}, the double series \eqref{epstein alternate} is estimated by the series $A\sum_{p=1}^\infty\frac{1}{p^{2\mathrm{Re}(s)}}$. Therefore, the double series \eqref{epstein alternate} converges uniformly when $s$ lies in any compact subset  in the half-plane Re$(s)>1/2$.

Furthermore, each partial sum in \eqref{partial sum} is an entire function of $s$. Hence, by Weierstrass' theorem for analytic functions, we conclude that the double series also represents an analytic function in the region Re$(s)>1/2$. This completes the proof  of the first part of the theorem.

The second part follows by assuming $\alpha,\beta\in\mathbb{R}\backslash\mathbb{Z}$ and taking $k=1$ in \eqref{general limit formula}, along with the identity
\begin{align*}
(\mathscr{D}_0\mathscr{F}_2)(w,w',\alpha,\beta)=\mathscr{F}_2(w;\alpha,\beta)-\mathscr{F}_2(w';\alpha,\beta),
\end{align*}
which is a direct consequence of the definition of the operator $\mathscr{D}_n$ defined in \eqref{doperator}. Note that setting $k = 1$ in \eqref{general limit formula} is valid, see proof of Theorem \ref{main} for more details.

\end{proof}

\begin{proof}[\textbf{Proof of Theorem \textup{\ref{main}}}] 
From \cite[pp.~33-34, Lemma 7(iv)]{VZ}, for any $T$, one has
$$\mathscr{D}_n\left(\frac{1}{T-x}\right)=\left(\frac{x-y}{(T-x)(T-y)}\right)^{n+1},$$
where the differential operator $\mathscr{D}_n$ is defined in \eqref{doperator}. Therefore, using the above result, it is easy to see that
\begin{align}\label{derivative of}
\left(\frac{(x-y)}{(px+q)(py+q)}\right)^k
=-\mathscr{D}_{k-1}\left(\frac{1}{p^{2k-1}(px+q)}\right).
\end{align} 
Equation \eqref{derivative of} along with \eqref{hf} and \eqref{epstein zeta} gives
\begin{align}
Z_Q(k; (\alpha, \beta) )=-(\mathscr{D}_{k-1}\mathscr{F}_{2k})(w,w';\alpha,\beta).\nonumber
\end{align}
Note that this equation is valid for every integer $k > 1$ when $\alpha$ and $\beta$ are arbitrary real numbers. However, observe that it is also valid for $k = 1$ only when $\alpha, \beta \in \mathbb{R} \setminus \mathbb{Z}$. This completes the proof of the theorem.
\end{proof}

\begin{proof}[\textbf{Proof of Theorem \textup{\ref{real second kronecker limit formula}}}]
The result follows once we observe that the summand on the right-hand side of \eqref{connection b/w real and our} is simply a special case of the zeta function in \eqref{epstein zeta}, corresponding to the case when $\alpha,\beta\in\mathcal{S}$. Then all the the parts of our theorem follows after applying Theorems \ref{mz} and \eqref{main} accordingly.
\end{proof}

%%%%%%%%%%%%%%%%%%%%%%
\medskip
%%%%%%%%%%%%%%%%%%%%%%

%%%%%%%%%%%%%%%%%%%%%%%%%%%%%%%%%%%%%%%%%%
\medskip
\subsection{Proof of cohomological aspects}
%%%%%%%%%%%%%%%%%%%%%%

\begin{proof}[\textbf{Proof of Proposition \textup{\ref{co}}}]
Since 
\begin{align*}
F(x)=\phi_S=\psi_{\alpha,\beta,2k},
\end{align*}
and 
\begin{align*}
G(x)=\phi_S.
\end{align*}
One checks directly that $\psi_{\alpha,\beta,2k}(x)$ satisfies \eqref{period}:
\begin{align*}
\psi_{\alpha,\beta,2k}|_{2k}(I+(-I)+S+(-S))=0,
\end{align*}
and
\begin{align*}
\psi_{\alpha,\beta,2k}|_{2k}(I+(-I)+U+(-U)+U^2+(-U)^2)=0.
\end{align*}
This claims  the proposition.
\end{proof}

\subsection{Proof of the connection between $\mathscr{F}_{k}(x; \alpha,\beta)$ and a generalized Dedekind eta function}
\begin{proof}[\textbf{Proof of Lemma \textup{\ref{new int repr}}}]
We have, for $c:=\mathrm{Re}(s)>0,\beta\in\mathbb{R}\backslash\mathbb{Z}$ \cite[p.~611, Formula 25.12.11]{nist},
\begin{align}
\frac{1}{2\pi i}\int_{(c)}\Gamma(s)\mathrm{Li}_s\left(e^{2\pi i\beta}\right)t^{-s}ds=\frac{1}{e^te^{-2\pi i\beta}-1},\nonumber
\end{align}
here, and in the rest of the paper, $\int_{(c)}ds$ denotes the line integral $\int_{c-i\infty}^{c+i\infty}ds$ with $c=\mathrm{Re}(s)$. Also for $d>0$,
\begin{align}
\frac{1}{2\pi i}\int_{(d)}\Gamma(s)x^{-s}t^{-s}ds=e^{-xt}.\nonumber
\end{align}
Thus, by invoking the Parseval formula \cite[p.~83, Equation (3.1.14)]{parsvel}, for $0<\lambda<1$, we obtain
\begin{align}
\int_0^\infty\frac{e^{-xt}}{e^{t}e^{-2\pi i\beta}-1}dt=\frac{1}{2\pi i}\int_{(\lambda)}\Gamma(s)\Gamma(1-s)\mathrm{Li}_{1-s}\left(e^{2\pi i\beta}\right)x^{-s}ds,\nonumber
\end{align}
which is, upon using the definition of $\psi_\beta(x)$ given in \eqref{poly},
\begin{align}\label{inverseMellinTransform of psi_beta}
\psi_\beta(x)-\frac{1}{x}=\frac{1}{2\pi i}\int_{(\lambda)}\Gamma(s)\Gamma(1-s)\mathrm{Li}_{1-s}\left(e^{2\pi i\beta}\right)x^{-s}ds.
\end{align}
The Lerch zeta function, defined in \eqref{Lerch},
%\begin{align}
%\phi(x,a,s):=\sum_{n=0}^\infty\frac{e^{2\pi inx}}{(n+a)^s}
%\end{align}
satisfies the following beautiful functional equation \cite{be}, for $0<a\leq1$ and $0<x<1$,
\begin{align}
\phi(x,a,1-s)=(2\pi)^{-s}\Gamma(s)\left\{e^{\frac{\pi is}{2}-2\pi iax}\phi(-a,x,s)+e^{-\frac{\pi is}{2}+2\pi ia(1-x)}\phi(a,1-x,s)\right\}.\nonumber
\end{align}
It is clear that $\phi(x,1,s)=e^{-2\pi ix}\mathrm{Li}_s\left(e^{2\pi ix}\right)$ and $\phi(\pm1,x,s)=\zeta(s,x)$, where $\zeta(s,x)$ is the Hurwitz zeta function. Therefore, if we let $a=1$ and replace $x$ by $\beta$ in the above functional equation then 
\begin{align}\label{realtion for li and hurwitz zeta}
\mathrm{Li}_{1-s}\left(e^{2\pi i\beta}\right)=(2\pi)^{-s}\Gamma(s)\left\{e^{\frac{\pi is}{2}}\zeta(s,\beta)+e^{-\frac{\pi is}{2}}\zeta(s,1-\beta)\right\},
\end{align}
for $0<\beta<1$. Substituting \eqref{realtion for li and hurwitz zeta} in \eqref{inverseMellinTransform of psi_beta}, we deduce that
\begin{align}\label{two intergals}
\psi_\beta(x)-\frac{1}{x}&=\frac{1}{2\pi i}\int_{(\lambda)}\Gamma(s)\Gamma(1-s)\Gamma(s)\left\{\zeta(s,\beta)(-2\pi ix)^{-s}+\zeta(s,1-\beta)(2\pi ix)^{-s}\right\}ds.
\end{align}
We next evaluate the line integrals involved in the above equation. To that end, note that, for $0<c'<1$, we have
\begin{align}
\frac{1}{2\pi i}\int_{(c')}\Gamma(s)\Gamma(1-s)(-2\pi ix)^{s-1}t^{-s}ds=\frac{i}{it+2\pi x}.\nonumber
\end{align}
By using the Cauchy residue theorem, one can find that, for $-1<c''<0$, 
\begin{align}\label{c''}
\frac{1}{2\pi i}\int_{(c'')}\Gamma(s)\Gamma(1-s)(-2\pi ix)^{s-1}t^{-s}ds=\frac{i}{it+2\pi x}-\frac{i}{2\pi x}.
\end{align}
We also need the following well-known result for the Hurwitz zeta function \cite[p.~609, Equation (25.11.25)]{nist}
\begin{align}\label{d}
\frac{1}{2\pi i}\int_{(d)}\Gamma(s)\zeta(s,\beta)t^{-s}ds=\frac{e^{-\beta t}}{1-e^{-t}},
\end{align}
which is valid for $d>1$. Again an application of Parsval formula \cite[p.~83, Equation (3.1.14)]{parsvel} on \eqref{c''} and \eqref{d} yields, for $1<c<2$,
\begin{align}\label{first int}
\frac{1}{2\pi i}\int_{(c)}\Gamma(s)\Gamma(1-s)\Gamma(s)\zeta(s,\beta)(-2\pi ix)^{-s}ds=\int_0^\infty\left(\frac{i}{it+2\pi x}-\frac{i}{2\pi x}\right)\frac{e^{-\beta t}}{1-e^{-t}}dt.
\end{align}
Similarly, we can also find
\begin{align}\label{second int}
\frac{1}{2\pi i}\int_{(c)}\Gamma(s)\Gamma(1-s)\Gamma(s)\zeta(s,1-\beta)(2\pi ix)^{-s}ds=\int_0^\infty\left(\frac{i}{it-2\pi x}+\frac{i}{2\pi x}\right)\frac{e^{-(1-\beta) t}}{1-e^{-t}}dt.
\end{align}
Adding \eqref{first int} and \eqref{second int}, for $1<c<2$,  we arrive at
\begin{align}
&\frac{1}{2\pi i}\int_{(c)}\Gamma(s)\Gamma(1-s)\Gamma(s)\left\{\zeta(s,\beta)(-2\pi ix)^{-s}+\zeta(s,1-\beta)(2\pi ix)^{-s}\right\}ds\nonumber\\
&=\int_0^\infty\left(\frac{i}{it+2\pi x}-\frac{i}{2\pi x}\right)\frac{e^{-\beta t}}{1-e^{-t}}dt+\int_0^\infty\left(\frac{i}{it-2\pi x}+\frac{i}{2\pi x}\right)\frac{e^{-(1-\beta) t}}{1-e^{-t}}dt\nonumber\\
&=\int_{-\infty}^0\left(\frac{i}{-it+2\pi x}-\frac{i}{2\pi x}\right)\frac{e^{\beta t}}{1-e^{t}}dt+\int_0^\infty\left(\frac{i}{it-2\pi x}+\frac{i}{2\pi x}\right)\frac{e^{\beta t}}{e^{t}-1}dt,\nonumber
\end{align}
where in the first integral we replaced $t$ by $-t$. Observe that the integrand of the above both integrals are same, thus we obtain
\begin{align}
&\frac{1}{2\pi i}\int_{(c)}\Gamma(s)\Gamma(1-s)\Gamma(s)\left\{\zeta(s,\beta)(-2\pi ix)^{-s}+\zeta(s,1-\beta)(2\pi ix)^{-s}\right\}ds\nonumber\\
&=\int_{-\infty}^\infty\left(\frac{i}{it-2\pi x}+\frac{i}{2\pi x}\right)\frac{e^{\beta t}}{e^{t}-1}dt.\nonumber
\end{align}
To substitute the above integral evaluation in \eqref{two intergals}, we need to shift the line of integration to $1<\lambda<2$ and use the residue theorem so as to find
\begin{align}
&\frac{1}{2\pi i}\int_{(\lambda)}\Gamma(s)\Gamma(1-s)\Gamma(s)\left\{\zeta(s,\beta)(-2\pi ix)^{-s}+\zeta(s,1-\beta)(2\pi ix)^{-s}\right\}ds\nonumber\\
&=-\frac{1}{2x}-\frac{i}{2\pi x}\left(\psi(\beta)-\psi(1-\beta)\right)+\int_{-\infty}^\infty\left(\frac{i}{it-2\pi x}+\frac{i}{2\pi x}\right)\frac{e^{\beta t}}{e^{t}-1}dt.\nonumber
\end{align}
Now theorem immediately follows upon substituting the above integral evaluation into \eqref{two intergals}.
\end{proof}

%%%%%%%%%%%%%%%%%%%%%%%%%%%%%%%
\begin{proof}[\textbf{Proof  of Proposition \textup{\ref{bi}}}] 
Employing \eqref{higher herglotz} and Lemma \ref{new int repr},
\begin{align}
\mathscr{F}_{k}(x; \alpha,\beta)&=\frac{1}{2x}\mathrm{Li}_{k}(e^{2\pi i \alpha})
-\frac{i}{2\pi x}\left(\psi(\beta)-\psi(1-\beta)\right)\mathrm{Li}_{k}(e^{2\pi i \alpha})\nonumber\\
&\quad+  \int_{0}^{i \infty} \left(\frac{1}{  t+ x}-\frac{1}{ x}\right)\sum_{n\geq 1} \frac{e^{2\pi i \alpha n}}{n^{k-1}}
\frac{e^{ 2\pi i \beta nt}}{1-e^{ 2\pi int}}dt\nonumber\\
&\quad+\int_{0}^{i \infty} \left(\frac{1}{ t- x}+\frac{1}{ x}\right)\sum_{n\geq 1} \frac{e^{2\pi i \alpha n}}{n^{k-1}}
\frac{e^{  2\pi (1- \beta) int}}{1-e^{  2\pi i nt} }dt.\nonumber
\end{align}
This implies that
\begin{align}\label{there}
\mathscr{F}_{k}(x; \alpha,\beta)+\mathscr{F}_{k}(x; \alpha,1-\beta)&=\frac{1}{x}\mathrm{Li}_{k}(e^{2\pi i \alpha})+\int_{0}^{i \infty} \left(\frac{1}{  t+ x}+\frac{1}{ t-x}\right)\sum_{n\geq 1} \frac{e^{2\pi i n\alpha+2\pi i \beta nt}}{n^{k-1}(1-e^{ 2\pi int})}dt\nonumber\\
&\quad+\int_{0}^{i \infty} \left(\frac{1}{  t+ x}+\frac{1}{ t-x}\right)\sum_{n\geq 1} \frac{e^{2\pi i n\alpha+2\pi i (1-\beta) nt}}{n^{k-1}(1-e^{ 2\pi int})}dt\nonumber\\
&=\frac{1}{x}\mathrm{Li}_{k}(e^{2\pi i \alpha})+\int_{0}^{i \infty} \left(\frac{1}{  t+ x}+\frac{1}{ t-x}\right)\sum_{\substack{n\geq 1\\m\geq0}} \frac{e^{2\pi i n\alpha+2\pi it(m+\beta)n}}{n^{k-1}}dt\nonumber\\
&\quad+\int_{0}^{i \infty} \left(\frac{1}{  t+ x}+\frac{1}{ t-x}\right)\sum_{\substack{n\geq 1\\m\geq0}} \frac{e^{2\pi i n\alpha+2\pi it(m+(1-\beta))n}}{n^{k-1}}dt\nonumber\\
&=\frac{1}{x}\mathrm{Li}_{k}(e^{2\pi i \alpha})+\int_{0}^{i \infty} \left(\frac{1}{  t+ x}+\frac{1}{ t-x}\right)\left\{A(t, 2-k,\beta, \alpha)\right.\nonumber\\
&\qquad\qquad\left.+A(t, 2-k, -\beta,\alpha) \right\}dt,
\end{align}
where in the last sterp we used the definition of $A(t,s,\beta,\alpha)$ in \eqref{gen ded} which becomes $A(t,s,\beta,\alpha)\newline=\sum_{m\geq1}\sum_{n\geq 1} n^{s-1}e^{2\pi i n\alpha}e^{2\pi it(m+\beta)n}$ for $0<\beta<1$.
 Now utilizing  \eqref{defnG} in \eqref{there} and then using the fact that $\mathscr{F}_{k}(x; \alpha,1-\beta)=\mathscr{F}_{k}(x; \alpha,-\beta)$, we conclude the proof of \eqref{cc}.
\end{proof}

\begin{proof}[\textbf{Proof of Equation \textup{\eqref{gde and}}}]
Replacing $\tau$ by $(m+\beta)\tau+\alpha$ in the well-known Lipschitz summation formula \cite{lip} 
\begin{align}
\sum_{n\in\mathbb{Z}}\frac{1}{(\tau+n)^s}=\frac{(-2\pi i)^s}{\Gamma(s)}\sum_{n=1}^\infty  n^{s-1}e^{2\pi in\tau}\quad (\mathrm{Re}(s)>1,\ \tau\in\mathbb{H}),\nonumber
\end{align}
we obtain
%\begin{align}\label{after wk}
%\sum_{n\in\mathbb{Z}}\frac{1}{(\tau+\alpha+n)^k}=\frac{(-2\pi i)^k}{(k-1)!}\sum_{n=1}^\infty  n^{k-1}e^{2\pi in\alpha}e^{2\pi in\tau}.
%\end{align}
%We again replace $\tau$ by $(m+\beta)\tau$ in \eqref{after wk} to deduce that
\begin{align}
\sum_{n\in\mathbb{Z}}\frac{1}{((m+\beta)\tau+\alpha+n)^s}=\frac{(-2\pi i)^s}{\Gamma(s)}\sum_{n=1}^\infty  n^{s-1}e^{2\pi in\alpha}e^{2\pi in(m+\beta)\tau}.\nonumber
\end{align}
Now if we sum over natural numbers $m$ including $0$ on the both sides of the above equation and use \eqref{gen ded}, then for $0<\beta<1$ and Re$(s)>2$,
\begin{align}\label{fourier}
\sum_{m=0}^\infty\sum_{n\in\mathbb{Z}}\frac{1}{((m+\beta)\tau+\alpha+n)^s}%=\frac{(-2\pi i)^k}{(k-1)!}\sum_{m=0}^\infty\sum_{n=1}^\infty  n^{k-1}e^{2\pi in\alpha}e^{2\pi in(m+\beta)\tau}.
=\frac{(-2\pi i)^s}{\Gamma(s)}A(\tau,s,\beta,\alpha).
\end{align} 
For $0<\beta<1$, equation \eqref{berndt es} can be re-written as 
\begin{align}
G(\tau,s,\beta,\alpha)&=\sum_{m=0}^\infty\sum_{n=-\infty}^\infty\frac{1}{((m+\beta)\tau+n+\alpha)^s}+\sum_{m=0}^\infty\sum_{n=-\infty}^\infty\frac{1}{((-m-1+\beta)\tau+n+\alpha)^s}\nonumber\\
&=  \sum_{m > -\beta}^\infty\sum_{n=-\infty}^\infty\frac{1}{((m+\beta)\tau+n+\alpha)^s}+\sum_{m>  \beta}^\infty\sum_{n=-\infty}^\infty\frac{1}{(-1)^s((m -\beta)\tau+n-\alpha)^s}\nonumber\\
&=\frac{(-2\pi i)^s}{\Gamma(s)}\left\{A(\tau,s,\beta,\alpha)+   e^{\pi is}A(\tau,s, -\beta,-\alpha)\right\}, \nonumber
\end{align}
where we invoked \eqref{fourier}. This proves the relation \eqref{gde and}.
\end{proof}

\subsection{ Proof of rational zeta values}
%%%%%%%%%%%%%%%%%%%%%%%%%%%%%%

\begin{proof}[\textbf{Proof of Lemma \textup{\ref{rz}}}]
We follow the notations from \cite{VZ}.  As in \eqref{cont}, take a ``reduced" number
$$w>1 >w'>0,$$ with purely periodic continued fraction
$(b_1, b_2, \cdots, b_r),$  so that the numbers in $Red(\mathscr{B})$ are exactly
$$w_k=(b_k, b_{k-1},\cdots, b_{1}).$$ 
Let $x$ be a reduced number in wide sense, that is, 
$$x>1, 0> x'>-1.$$ Then it is known that
\begin{eqnarray} 
x=  a_1+\cfrac{1}{a_{2}+\ \genfrac{}{}{0pt}{0}{}{\ddots\genfrac{}{}{0pt}{0}{}\ \genfrac{}{}{0pt}{0}{}{+\cfrac{1}{a_m+\cfrac{1}{a_{1}+\genfrac{}{}{0pt}{0}{}{\ddots}}}}}} , \,\, (a_i\in \mathbb{Z}, a_i\geq 1, j\neq 1)\nonumber
\end{eqnarray}
So   a wide ideal class $\mathscr{A} $ of a real quadratic field $K$ contains a cycle of  numbers
\begin{eqnarray} 
x_i=  a_i+\cfrac{1}{a_{i+1}+\ \genfrac{}{}{0pt}{0}{}{\ddots\genfrac{}{}{0pt}{0}{}\ \genfrac{}{}{0pt}{0}{}{+\cfrac{1}{a_{i+m}+\cfrac{1}{a_{1}+\genfrac{}{}{0pt}{0}{}{\ddots}}}}}} , \,\, ( i\in \mathbb{Z} / m\mathbb{Z}).\nonumber
\end{eqnarray}
Note that 
\begin{eqnarray*}
&&w_1=1+x_1=a_1+2-\frac{1}{w_2},\\
&&w_2=2-\frac{1}{w_3}, \cdots, w_{a_2}=2-\frac{1}{w_{a_2+1}},\\
&& w_{a_2+1}=x_3+1=a_3+2-\frac{1}{w_{a_2+2}}, \cdots.
\end{eqnarray*}
Let Red$(\mathscr{B})$ be a narrow ideal class $\mathscr{B}$ so
Red$_w(\mathscr{B})=\{x_1, x_3, \cdots\}$ and 
Red$_{w}(\mathscr{B}^*)=\{x_2,x_4, \cdots\}.$
\begin{eqnarray*}
\sum_{j=2}^{a_2+1}F(w_j, w_j', \alpha, \alpha)
&=&\sum_{j=1}^{a_2+1}
\left[
g(w_j-1, w_j'-1,0, \alpha) - g\left(1-\frac{1}{w_j}, 1-\frac{1}{{w_j}'},0, \alpha\right)\right]\\
&=& g ( w_{a_2+1}-1,  w_{a_2+1}'-1,0, \alpha)- g\left(1-\frac{1}{w_2}, 1-\frac{1}{{w_2}'},0,\alpha\right)\\
&=& g(x_3, x_3', 0, \alpha)-g\left(\frac{1}{x_2}, \frac{1}{x_2'}, 0,\alpha \right)
\end{eqnarray*}
Summing this cyclically gives a proof of Lemma 1.
From the choice of $g$ we get
$$F(x,y,\alpha,\alpha)=H_1(x-y,0,\alpha) - H_1\left(-\frac{1}{x}+\frac{1}{y}, 0,\alpha\right) $$
and also using the fact
$$w_j-w_j'=-\frac{1}{w_{j+1}}+\frac{1}{w_{j+1}'},$$
we conclude the proof of Lemma.
\end{proof}

 \medskip

%%%%%%%%%%%%%%%%%%%%%%%%%%%%%%%%%%%
 \begin{proof}[Proof of Corollary \textup{\ref{rc}}]
From Lemma \ref{rz} 
\begin{align*}
\mathscr{I}(\mathscr{B})&=\sum_{w\in \mathrm{Red}(\mathscr{B})}F(w, w', \alpha, \alpha)\\
&=\sum_{x\in \mathrm{Red}(\mathscr{B})}g(x,x',0, \alpha)-
\sum_{x\in \mathrm{Red}(\mathscr{B}^*)}g\left(\frac{1}{x}, \frac{1}{x'}, 0,\alpha\right)
\end{align*}
the conclusion follows.
\end{proof}

\medskip

\begin{proof}[\textbf{Proof of Theorem \textup{\ref{v-z}}}]
%\blue
Let us define, for $x>0,$
\begin{align}\label{v0}
V_0(x;\alpha,\beta)&:= -x^{2k-2}\mathscr{F}_{2k}\left( \frac{1}{x}; \beta, \alpha\right)  +\frac{1}{4x}\mathrm{Li}_{2k}(e^{2\pi i \alpha})+\frac{1}{4}x^{2k-1} \mathrm{Li}_{2k}(e^{2\pi i\beta})\nonumber\\
&=\mathscr{F}_{2k}\left(x;\alpha,\beta \right) -\frac{3}{4x} \mathrm{Li}_{2k}(e^{2\pi i \alpha})-\frac{3}{4}x^{2k-1} \mathrm{Li}_{2k}(e^{2\pi i   \beta})\nonumber\\
&\qquad-\sum_{r=1}^{2k-1}(-x)^{r-1}\mathrm{Li}_{2k-r}(e^{2\pi i \alpha})\mathrm{Li}_r(e^{2\pi i \beta}),
\end{align}
where the last step follows upon employing \eqref{two term}.  Let us also define
\begin{align}
&V(x;\alpha,\beta):=V_0(|x|;\alpha,\beta)+\frac{3}{4x} 
  \mathrm{Li}_{2k}(e^{2\pi i \alpha})  
+ \frac{3}{4}x^{2k-1} \mathrm{Li}_{2k}(e^{2\pi i   \beta}), \,\,
x\in \mathbb{R}\backslash \{0\}.\nonumber
%& \mathscr{F}_{2k}\left(x-1; \alpha-\beta, \beta \right)=V(x-1; \alpha-\beta, \beta)+\frac{1}{x-1} Li_k(e^{2\pi i( \alpha-\beta)})\\&-Li_k(e^{2\pi i \beta}) \frac{1}{x-1}  - Li_{2k}(e^{2\pi i ( \alpha-\beta)})  (x-1)^{2k-1} \\&-(x-1)^{2k-1}Li_k(e^{2\pi i \beta})-\sum_{r=1}^{2k-1}(-x+1)^{r-1}Li_{k-r}(e^{2\pi i( \alpha-\beta)})Li_r(e^{2\pi i \beta}), \\& (x-1)^{2k-2}\mathscr{F}_{2k}\left( \frac{x}{x-1}, \alpha, \alpha-\beta  \right) \\ &:=-x^{2k-2}V (\frac{x-1}{x},\alpha-\beta, \alpha)\\&+   \mathrm{Li}_{2k}(e^{2\pi i (\alpha-\beta)})  \frac{x^{2k-1}}{x-1}+  \mathrm{Li}_{2k}(e^{2\pi i   \alpha)}) \frac{(x-1)^{2k-1}}{x}  \nonumber
\end{align}
 So, by invoking two and three-term functional equations \eqref{two term} and  \eqref{three term}, we obtain
\begin{align*}
  \mathscr{F}_{2k}\left(x;\alpha ,\beta\right)&=
 V(x-1; \alpha-\beta, \beta) - x^{2k-2}V \left(\frac{x-1}{x};\alpha-\beta, \alpha\right)+\frac{x^{2k-1}-1}{x-1}\mathrm{Li}_{2k}\left(e^{2\pi i(\alpha-\beta)}\right)\\
 &\quad-\sum_{r=1}^{2k-1}(-1)^{r-1}(x-1)^{r-1}\mathrm{Li}_{r,2k-r}\left(e^{2\pi i\beta},e^{2\pi i(\alpha-\beta)}\right)\\
 &\quad+\sum_{r=1}^{2k-1}(-1)^{r-1}(x-1)^{r-1}\mathrm{Li}_{2k-r}\left(e^{2\pi i(\alpha-\beta)}\right)\mathrm{Li}_r\left(e^{2\pi i\beta}\right)\\
 &=V(x-1; \alpha-\beta, \beta) - x^{2k-2}V \left(\frac{x-1}{x};\alpha-\beta, \alpha\right)\\
&\quad+\mathrm{(polynomial\ in}\ x \ \mathrm{of\ degree} \leq 2k-2).
\end{align*}
%by invoking the two and three-term functional equations \eqref{two term} and  \eqref{three term}. 
%in the second step below, we see that\begin{eqnarray*}&&  \mathscr{F}_{2k}\left(x,\alpha ,\beta\right) = V(x-1,\alpha,\beta)-x^{2k-2}V\left(\frac{x-1}{x},\alpha,\beta\right)\\&& \quad+\mbox{(polynomial in $x$ with degree $\leq 2k-2$)}.\end{eqnarray*}
Polynomials of degree $\leq 2k-2$ belong to the kernel of the operator $\mathscr{D}_{k-1}$ (see \cite{VZ}), therefore, upon letting $\alpha=\beta$, we obtain
\begin{align}
\left(\mathscr{D}_{k-1}\mathscr{F}_{2k}\right)\left(x,y; \alpha ,\alpha\right)=\left(\mathscr{D}_{k-1}V\right)(x-1,y-1;0,\alpha)-\left(\mathscr{D}_{k-1}V\right)\left(\frac{x-1}{x},\frac{y-1}{y};0,\alpha\right).\nonumber
\end{align}
Hypothesises of Lemma \ref{rz} are satisfied upon letting $F(x,y,\alpha,\alpha)=\left(\mathscr{D}_{k-1}\mathscr{F}_{2k}\right)\left(x,y; \alpha ,\alpha\right)$ and $g(x,y,0,\alpha)=(\mathscr{D}_{k-1}V)(x,y,0,\alpha)$. Thus, Theorem \ref{main} along with \eqref{deffxy} imply that $\mathscr{I}(\mathscr{B})=D^{k/2}\zeta(k,(\alpha,\alpha),\mathscr{B})$. Using the fact
 $\mathscr{D}_{k-1}(x^{2k-1})=(x-y)^k$ and $\mathscr{D}_{k-1}(\frac{1}{x})=-\left(-\frac{1}{x}+\frac{1}{y}\right)^{k}$ (see \cite{VZ})  replace $g(x,y,0, \alpha)$ by 
 $G_0(x,y,0,\alpha) = \mathscr{D}_{k-1}(V_0(|x|, 0,\alpha).$ Using \eqref{v0}, we have
 \begin{align*}
 &V_0(|x|;0,\alpha)+x^{2k-2}V_0\left(\frac{1}{|x|};0, \alpha\right)\\
  &=\mathscr{F}_{2k}\left(|x|;0, \alpha\right)+|x|^{2k-2}\mathscr{F}_{2k}\left(\frac{1}{|x|};0, \alpha\right)-\frac{3}{4}\left(\frac{1}{|x|}+|x|^{2k-1}\right)\left(\zeta(2k)+\mathrm{Li}_{2k}(e^{2\pi i   \alpha})\right)\\
&\quad-\sum_{r=0}^{ 2k-1 }(-1)^{ r-1}\left(|x|^{r-1}+|x|^{2k-r-1}\right) \zeta(2k-r) \mathrm{Li}_{r}(e^{2\pi i\alpha}).
\end{align*}
Applying the operator $\mathscr{D}_{k-1}$ on both sides of the above equation gives us
 \begin{align*}
 &\mathscr{D}_{k-1}(V_0)(|x|,|y|;0,\alpha)+(-1)^{k-1}\mathscr{D}_{k-1}(V_0)\left(\frac{1}{|x|},\frac{1}{|y|};0, \alpha\right)\\
&=\mathscr{D}_{k-1}\left(\mathscr{F}_{2k}\left(|x|;0, \alpha\right)+|x|^{2k-2}\mathscr{F}_{2k}\left(\frac{1}{|x|};0, \alpha\right)-\frac{3}{4}\left(\frac{1}{|x|}+|x|^{2k-1}\right)\left(\zeta(2k)+\mathrm{Li}_{2k}(e^{2\pi i   \alpha})\right)\right.\\
&\left.\quad+\sum_{r=0}^{ 2k-1 }\left(|x|^{2r-1}+|x|^{2k-2r-1}\right) \zeta(2k-2r) \mathrm{Li}_{2r}(e^{2\pi i\alpha})\right).
\end{align*}
Thus, $G_0(x,y,0, \alpha) + (-1)^{k-1} G_0(\frac{1}{x}, \frac{1}{y},0,\alpha)=
 W_k(x,y, \alpha , \alpha).$ Result now follows from Corollary \ref{rc}. 
\end{proof}

\noindent  {\bf{Acknowledgements:}} We sincerely thank the referee for providing valuable suggestions and insightful corrections, which have greatly improved the exposition and clarity of our work. Authors would like to express their gratitude to Prof. Don Zagier for providing valuable comments regarding the well-definedness of $\zeta(s, (\alpha, \beta), \mathscr{B} ) $ in \eqref{qdr}. 
The first author was partially supported by   NRF $2022R1A2B5B01001871 , 2021R1A6A1A1004294412$ and the second author  was partially supported by the Grant ANRF/ECRG/2024/003222/PMS of the Anusandhan National Research Foundation (ANRF), Govt. of India, and CPDA and FIG grants of IIT Roorkee. Both the authors thank these institutions for the support.
%Part of the work was done when the second author was a post-doctoral fellow at the IBS center for Geometry and Physics, POSTECH, South Korea, and he wants to acknowledge the support of the grant IBS-R003-D1 of the IBS-CGP, POSTECH, South Korea. Authors thank these institutions for the support.

\noindent
\textbf{Conflict of interest statement:} On behalf of all authors, the corresponding author states that there is no conflict of interest.

\end{document}